\documentclass{scrartcl}

\usepackage{gd, e-jc}
\dateline{}{}{}
\title{Combinatorial Game Distributions of Steiner Systems}
\author{Yuki Irie\thanks{This paper is a substantially revised version of Section 2.A of the author's Ph.D. thesis, which was prepared under the supervision of Masaaki Kitazume for Chiba University. This revision includes adding a result that characterizes projective Steiner triple systems.}\\
\small  Research Alliance Center for Mathematical Sciences\\[-0.8ex]
\small Tohoku University\\[-0.8ex]
\small Miyagi, Japan\\
\small\texttt {yirie@tohoku.ac.jp}\\
}
\Copyright{}
\date{}
\date{}
\title{Combinatorial Game Distributions of Steiner Systems}
\begin{document}
\nolinenumbers
\maketitle
\begin{abstract}
The \textbf{\(\sP\)-position sets} of some combinatorial \textbf{games have} special combinatorial \textbf{structures}.
For example, the \textbf{\(\sP\)-position} set of the hexad game, first investigated by Conway and Ryba, is the block set of the Steiner system \(S(5, 6, 12)\) in the shuffle numbering, \textbf{denoted by \(\shD\)}.
There were, however, few known games related to Steiner systems like the hexad game.
\textbf{For a given Steiner system, we construct a game whose \(\sP\)-position set is its block set.}

By using constructed games, we obtain the following two results.
First, we characterize \textbf{\(\shD\) among the 5040 isomorphic \(\ccS(5, 6, 12)\) with point set \(\set{0, 1, \dotsc, 11}\).}
\textbf{For each \(\ccS(5, 6, 12)\), our construction produces a game whose \(\sP\)-position set is its block set.}
\textbf{From \(\shD\), we obtain the hexad game, and this game is characterized as a unique game with the minimum number of positions among the obtained 5040 games.}
Second, we characterize projective Steiner triple systems by using game distributions.
Here, the game distribution of a Steiner system \(\cD\) is the frequency distribution of the numbers of positions in games obtained from \textbf{Steiner systems isomorphic to} \(\cD\).
\textbf{We find that the game distribution of an \(\ccS(\cct, \cct + 1, \ccv)\) can be decomposed into symmetric components and that a Steiner triple system is projective if and only if its game distribution has a unique symmetric component.}
\end{abstract}

\section{Introduction}
\label{sec:orgcf73c6d}
\label{orgaae31b4}
\mylink{sec-intro}The starting point of this paper is a two-player \textbf{(impartial)} game called the \emph{hexad game} \cite{conway-lexicographic-1986, kahane-hexad-2001}.\footnote{\textbf{According to \cite{kahane-hexad-2001}, the hexad game was first investigated by Conway and Ryba}.}
In combinatorial game theory, one of the fundamental goals is to determine the \textbf{\(\sP\)-position} set of a game,
\textbf{where a \(\sP\)-position is a winning position for the Previous player to move}.
The hexad game is notable because its \textbf{\(\sP\)-position} set, which is a subset of \(\binom{[12]}{6}\), is the block set of a Steiner system \(S(5, 6, 12)\) \cite{conway-lexicographic-1986, kahane-hexad-2001},
where \([12] = \set{0, 1, \dotsc, 11}\) and \({[12] \choose 6} = \bigfamily{\ccP \subset [12] : \size{\ccP} = 6}\).
\textbf{Here, a pair \(\bigl([\ccv], \Block\bigr)\) is called a \emph{Steiner system \(S(\cct, \cck, \ccv)\) with point set \([v]\) and block set \(\Block\)} if \(\Block \subseteq {[\ccv] \choose \cck}\) and every \(\ccT \in {[\ccv] \choose \cct}\) is contained in a unique \(\ccB \in \Block\).}
Further, Conway and Sloane \cite{conway-lexicographic-1986} found a game whose \textbf{\(\sP\)-position} set is the block set of a Steiner system \(S(5, 8, 24)\).
No methods, \textbf{however}, were known to find games related to Steiner systems like the above games.

\mylink{p2-4}For a given Steiner system, we construct a game whose \textbf{\(\sP\)-position} set is equal to its block set \textbf{(Corollary \ref{orgc92db7a})}.\footnote{\textbf{Since we start with a \(\sP\)-position set, our approach may be referred to as reverse game theory or mechanism design (see \cite{fraenkel-games-2019})}.}
\yon
We here illustrate our construction using the \(\ccS(1, 2, 4)\) with block set \(\bigfamily{\set{0, 2}, \set{1, 3}}\); 
our aim is to obtain the game shown in Figure \ref{fig:org4b96216}.
We use a digraph to represent a game.
Let \(\Gamma\) be a digraph with vertex set \({[4] \choose 2}\) and edge set
\[
 \biggset{\bigl(\ccP, \ccP^{(\ccp\ \ccq)} \bigr) \in {[4] \choose 2}^2 : \ccq < \ccp,\ \ccp \in \ccP,\ \ccq \not \in \ccP},
\]
where \((\ccp\ \ccq)\) is the transposition of \(\ccp\) and \(\ccq\).
We denote by \(\Position{\Gamma}\) and \(\Edge{\Gamma}\) the vertex and edge sets of \(\Gamma\), respectively.
For example, \(\bigl(\set{2, 3}, \set{2, 3}^{(3\ 0)}\bigr) = \bigl(\set{2, 3}, \set{0, 2}\bigr) \in \Edge{\Gamma}\).
\yoff
For \(\ccP \in \Position{\Gamma}\), let \(\InNeighbors{\ccP}[\Gamma]\) and \(\OutNeighbors{\ccP}[\Gamma]\) be the sets of in-neighbors and out-neighbors of \(\ccP\), respectively.
\textbf{That} is,
\[
 \InNeighbors{\ccP}[\Gamma] = \bigset{\ccR \in \Position{\Gamma} : (\ccR, \ccP) \in \Edge{\Gamma}} \quad \tand \quad
 \OutNeighbors{\ccP}[\Gamma] = \bigset{\ccQ \in \Position{\Gamma} : (\ccP, \ccQ) \in \Edge{\Gamma}}.
\]
For example, \(\InNeighbors{\ymod{\set{0, 2}}}[\Gamma][\bigcbrackets] = \bigfamily{\set{1, 2}, \set{0, 3}, \set{2, 3}}\).
Define
\begin{equation}
\label{equ:def-bposition}
 \bPosition{\Block}[\Gamma] = \Block \cup \bigcup_{\ccB \in \Block} \InNeighbors{\ccB}[\Gamma],
\end{equation}
\textbf{where \(\Block = \bigfamily{\set{0, 2}, \set{1, 3}}\)}.
\textbf{Then \(\bPosition{\Block}[\Gamma] = {[4] \choose 2} \setminus \bigfamily{\set{0, 1}}\).}
Let \(\Induced{\Gamma}{\ymod{\bPosition{\Block}[\Gamma]}}[\bigcbrackets]\) denote the subgraph induced in \(\Gamma\) by \(\ymod{\bPosition{\Block}[\Gamma]}\), that is, the digraph with vertex set \(\ymod{\bPosition{\Block}[\Gamma]}\) and edge set
\[
  \Bigset{(\ccP, \ccQ) \in \Edge{\Gamma}: \ccP, \ccQ \in \ymod{\bPosition{\Block}[\Gamma]}}.
\]
\textbf{We consider the digraph \(\Induced{\Gamma}{\bPosition{\Block}[\Gamma]}[\bigcbrackets]\) as the following two-player game.}
Before \textbf{a} game begins, we put a token on a \textbf{starting vertex} \(\ccP \in \bPosition{\ymod{\Block}}[\Gamma]\).
The first player moves the token from \(\ccP\) to an out-neighbor \(\ccQ\) of \(\ccP\).
Similarly, the second player moves it from \(\ccQ\) to an out-neighbor \(\ccR\) of \(\ccQ\).
In this way, the two players alternately move the token from the current \textbf{vertex} to \textbf{one of its out-neighbors}.
The player who moves the token last wins. 
\yon\ 
In this manner, we can view a digraph as a two-player game.
Then the \(\sP\)-position set of \(\Induced{\Gamma}{\bPosition{\Block}[\Gamma]}[\bigcbrackets]\) is \(\Block\).
\begin{figure}[H]
\centering
\includegraphics[scale=0.52]{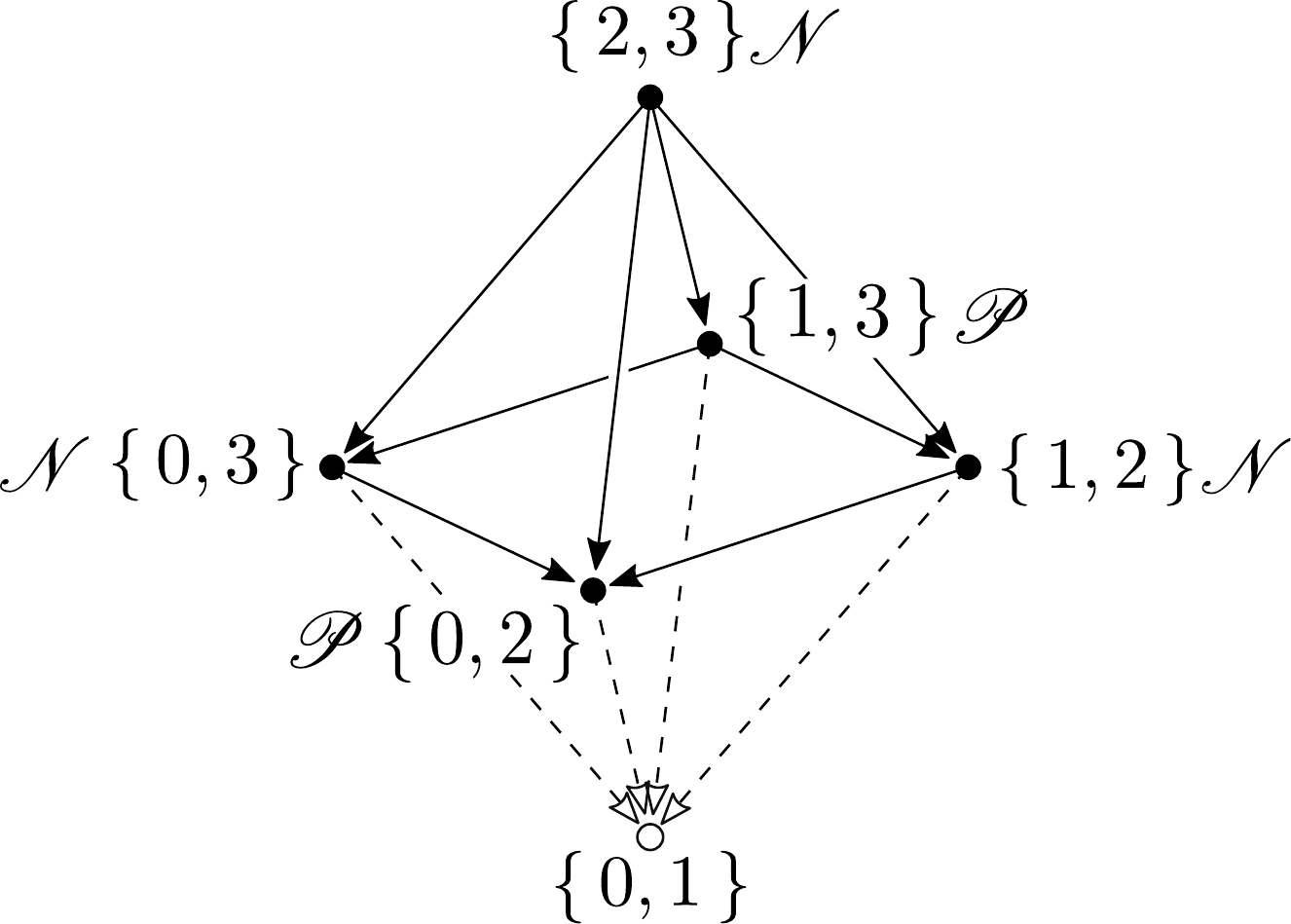}
\caption{\label{fig:org4b96216}\textbf{A game whose \(\sP\)-position set is \(\bigfamily{\set{0, 2}, \set{1, 3}}\). The labels \(\sP\) and \(\sN\) indicate winning positions for the previous and next players, respectively.}}
\end{figure}
\noindent
For example, if \(\set{1, 3}\) is the starting vertex,
then the first player can move the token only to \(\set{0, 3}\) or \(\set{1, 2}\),
so the second player can move it to \(\set{0, 2}\) and win.
Thus \(\set{1, 3}\) is a winning position for the second player (the previous player to move), that is, a \(\sP\)-position.
In the same way, for a given Steiner system, we can construct a game whose \(\sP\)-position set is its block set (Corollary \ref{orgc92db7a}).

In this paper, we investigate the frequency distribution of the numbers of positions of games obtained from a Steiner system.
Let \(\cD\) be an \(\ccS(\cct, \cck, \ccv)\) with point set \([\ccv]\), and let \(\ccN\)
be the index of the automorphism group of \(\cD\) in the symmetric group \(\Sym([\ccv])\) on \([\ccv]\).
Then there are \(\ccN\) Steiner systems \(\ccS(\cct, \cck, \ccv)\) with point set \([\ccv]\) isomorphic to \(\cD\),
and, for each of them, we can construct a game whose \(\sP\)-position set is its block set.
The frequency distribution of the numbers of positions (vertices) of the obtained \(\ccN\) games is called the \emph{game distribution} of \(\cD\).
For example, there are three \(\ccS(1, 2, 4)\) with point set \([4]\), and their block sets are as follows:
\[
 \Block = \bigfamily{\set{0, 2}, \set{1, 3}}, \quad \Block' = \bigfamily{\set{0, 1}, \set{2, 3}}, \quad \Block'' = \bigfamily{\set{0, 3}, \set{1, 2}}.
\]
As we have seen, \(\size{\bPosition{\Block}[\Gamma]} = 5\).
A direct computation shows that
\[
 \msize{\bPosition{\Block'}[\Gamma]} = \size{{[4] \choose 2}} = 6 \quad \tand \quad \msize{\bPosition{\Block''}[\Gamma]} = \size{{[4] \choose 2} \setminus \bigfamily{\set{0, 1}, \set{0, 2}}} = 4.
\]
Therefore the game distribution of an \(\ccS(1, 2, 4)\) is as follows:
\begin{center}
\begin{tabular}{c|ccc|c}
\(\#\) positions & 4 & 5 & 6 & Total\\
\hline
\(\#\) games & 1 & 1 & 1 & 3\\
\end{tabular}
\end{center}

By using game distributions, we present the following two results.
\yoff

First, we characterize the Steiner system \(S(5, 6, 12)\) in the shuffle numbering, \textbf{denoted by \(\shD\)}.
It is well-known that \(S(5, 6, 12)\) is unique up to isomorphism \cite{witt-ueber-1937} and its automorphism group is the sporadic simple Mathieu group \(M_{12}\) \cite{witt-5fach-1937}.
\yon \ 
Since the index of \(\ccM_{12}\) in \(\Sym([12])\) equals 5040, it follows that
there are 5040 \(S(5, 6, 12)\) with point set \([12]\), and our construction produces 5040 games.
Here, \(\shD\) is one of the 5040 \(S(5, 6, 12)\).
The name of the shuffle numbering comes from the fact that the automorphism group of \(\shD\) is
generated by the two permutations \(i \mapsto 11 - i\) and \(\cci \mapsto \min(2 \cci, 23 - 2 \cci)\) of \([12]\), called the Mongean shuffle, 
and \(\shD\) has many interesting properties \cite{conway-sphere-2013}.
In fact, the \(\sP\)-position set of the hexad game is the block set of \(\shD\) \cite{conway-lexicographic-1986, kahane-hexad-2001}.
Conversely, our construction produces the hexad game from \(\shD\).
Surprisingly, this game can be characterized as a unique game with the minimum number of positions among the obtained 5040 games;
this result is easily seen by using the game distribution of an \(\ccS(5, 6, 12)\) (Theorem~\ref{org0f36e55}).
\yoff
In particular, we can characterize \textbf{\(\shD\)} by using games.

Second, we characterize projective Steiner triple systems among all Steiner triple systems.
In this paragraph, we consider non-isomorphic designs unlike in the previous paragraph.
The projective geometry \textbf{\(\text{PG}(d, 2)\)} forms a Steiner system \(S(2, 3, 2^{d + 1} - 1)\), called the \emph{projective Steiner triple system of order} \(2^{d + 1} - 1\) (Example \ref{org0d144dd}).
It is known that there are 80 non-isomorphic Steiner systems \(S(2, 3, 15)\) \cite{cole-complete-1917}, and one of them is projective.
How can we characterize the projective one among them?
In combinatorial design theory, characterizations of particular designs, including projective Steiner triple systems, have been studied (see, for example, \cite{colbourn-triple-1999, dehon-designs-1977, doyen-ranks-1978, veblen-projective-1916, baartmans-binary-1996}).
In this paper, we show that \textbf{the game distribution of an \(S(\cct, \cct + 1, \ccv)\)} can be decomposed into \emph{symmetric components} (Proposition \ref{org3399301}). 
We then prove that an \(S(2, 3, \ccv)\) is projective if and only if its game distribution has a unique symmetric component (Theorem \ref{org1705e9e}).

This paper is organized as follows. In Section \ref{org75d1671}, we recall the concepts of \textbf{a} Steiner system and \textbf{an} impartial game.
Section \ref{org887ee75} contains the construction of games whose \textbf{\(\sP\)-position} set equals the block set of a Steiner system.
In Section \ref{org927a8a4}, we introduce game distributions and then characterize the \(S(5, 6, 12)\) in the shuffle numbering.
In Section \ref{orgee4c4c8}, it is shown that the game distribution of an \(S(\cct, \cct + 1, \ccv)\) can be decomposed into symmetric components.
We finally characterize projective Steiner triple systems by using game distributions in Section \ref{orge91698d}.

\section{Preliminaries}
\label{sec:org2397e72}
\label{org75d1671}
\textbf{We recall some facts about Steiner systems and impartial games that will be used later.}
\subsection{Steiner Systems}
\label{sec:orgbb2fdcd}
\label{orgdf895f0}
\mylink{sec-steiner}Let \(\cct\) be a nonnegative integer and \textbf{let} \(\ccv\), \(\cck\), and \(\lambda\) be positive integers with \(\cct \le \cck \le \ccv\).
A \emph{\(\design{\cct}{\ccv, \cck, \lambda}\)} \(\cD\) is a pair \((\Vertex{\cD}, \Block[\cD])\) where \(\Vertex{\cD}\) is a \textbf{set with \(\size{\Vertex{\cD}} = \ccv\)}
and \(\Block[\cD]\), \textbf{called the \emph{block set} of \(\cD\), is a subset of \({\Vertex{\cD} \choose \cck}\)} such that
every \textbf{\(\ccT \in {\Vertex{\cD} \choose \cct}\)} is contained in exactly \(\lambda\) blocks.
\textbf{Note that a \(\design{\cct}{\ccv, \cck, 1}\) is a Steiner system \(\ccS(\cct, \cck, \ccv)\).}
Throughout this paper, we assume that \(\Vertex{\cD} = [\ccv] = \set{0, 1, \dotsc, \ccv - 1}\) unless otherwise noted.
If \(\cD\) is a \(\design{\cct}{\ccv, \cck, \lambda}\), then it is also an \(\design{\cci}{\ccv, \cck, \lambda_\cci}\) for \(0 \le \cci \le \cct\), where
\begin{equation}
\label{equ:lambdai}
 \lambda_\cci = \frac{{\ccv - \cci \choose \cct - \cci}}{{\cck - \cci \choose \cct - \cci}} \lambda.
\end{equation}
Note that \(\lambda_0\) equals the number of blocks of \(\cD\).

 \begin{example}
 \comment{Exm.}
\label{sec:org4cdd8a1}
Let \(\Block = \bigfamily{\set{0, 1}, \set{2, 3}, \dotsc, \set{2 \ccv - 2, 2 \ccv - 1}}\).
Then \(([2 \ccv], \Block)\) is an \(S(1, 2, 2 \ccv)\).
 
\end{example}

 \begin{example}[\ymod{P}rojective Steiner triple systems]
 \comment{Exm. [\ymod{P}rojective Steiner triple systems]}
\label{sec:orge97801f}
\label{org0d144dd}
Let \(\FF_2\) be the finite field with two elements.
Let 
\[
 \Vertex{} = \FF_2^{\ccd + 1} \setminus \bigfamily{(0, \dotsc, 0)}\  \tand\ \Block = \biggfamily{\set{\ccp, \ccq, \ccp + \ccq} \in {\Vertex{} \choose 3} : \set{\ccp, \ccq} \in {\Vertex{} \choose 2}}.
\]
We see that \((\Vertex{}, \Block)\) is a Steiner system \(S(2, 3, 2^{\ccd + 1} - 1)\).
This design is called the \emph{projective Steiner triple system of order} \(2^{\ccd + 1} - 1\),
\textbf{since \(\Vertex{}\) and \(\Block\) can be regarded as the point and line sets of the projective geometry \(\text{PG}(\ccd, 2)\)}.
For example, if \(\ccd = 2\), then the blocks are as shown in Figure \ref{fig:orgb8b0773}.

\begin{figure}[htbp]
\centering
\includegraphics[scale=0.52]{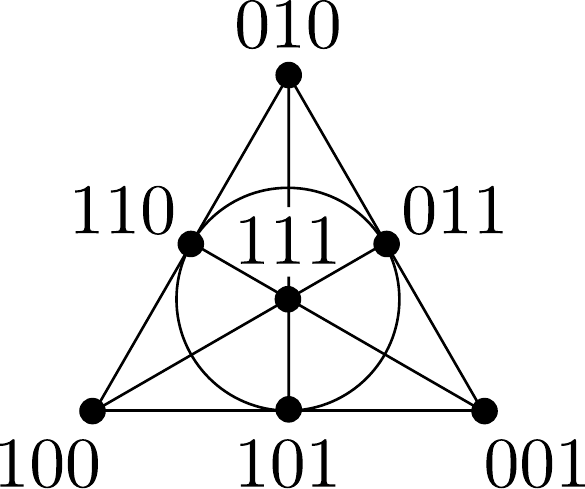}
\caption{\label{fig:orgb8b0773}The blocks of the projective Steiner triple system of order 7. For example, \(\bigfamily{(1, 0, 0), (1, 1, 0), (0, 1, 0)}\) and \(\bigfamily{(1, 1, 0), (1, 0, 1), (0, 1, 1)}\) are blocks.}
\end{figure}
 
\end{example}

\comment{Isomorphisms}
\label{sec:org86140d0}
We next define isomorphisms.
Let \(\cD\) and \(\cD'\) be \(\designs{t}{v, k, \lambda}\).
A bijection \(\phi \colon \Vertex{\cD} \to \Vertex{\cD'}\) is called an \emph{isomorphism} if
\(\phi(\ccB) \in \Block[\cD']\) for every \(\ccB \in \Block[\cD]\), where \(\phi(\ccB) = \set{\phi(\ccb) : \ccb \in \ccB}\).
If there is an isomorphism between \(\cD\) and \(\cD'\), then they are said to be \emph{isomorphic}.
An isomorphism from \(\cD\) to \(\cD\) is called an \emph{automorphism} of \(\cD\).
Let \(\Aut(\cD)\) denote the group of automorphisms of \(\cD\).

 \begin{example}
 \comment{Exm.}
\label{sec:org0b756e2}
\label{org8d07afb}
Let \(\cD\) be the \(S(1, 2, 4)\) with block set \(\bigfamily{\set{0, 1}, \set{2, 3}}\).
The automorphism group of \(\cD\) is generated by the two permutations 
\((0\ 1)\) and \((0\ 2)(1\ 3)\). This group is isomorphic to the dihedral group \(D_8\) of order 8.
It is easy to see that \(S(1, 2, 4)\) is unique up to isomorphism.
In general, \(S(1, 2, 2v)\) is unique up to isomorphism.
 
\end{example}

\begin{remark}
 \comment{Rem.}
\label{sec:org0eb0e9b}
\label{orgde64e0f}
\mylink{p4-21}
A Steiner system \(S(2, 3, \ccv)\) is called a \emph{Steiner triple system} (see \cite{colbourn-triple-1999} for details).
Kirkman \cite{kirkman-problem-1847} showed that
an \(S(2, 3, \ccv)\) exists if and only if \(\ccv \equiv 1, 3 \pmod{6}\).
Although \(S(2, 3, 7)\) and \(S(2, 3, 9)\) are unique up to isomorphism,
\(S(2, 3, \ccv)\) is not unique for \(\ccv > 9\) \cite{depasquale-sui-1899, moore-concerning-1893}. For example,
there are 80 non-isomorphic \(S(2, 3, 15)\) \cite{cole-complete-1917} and one of them is projective.
Let \(\cD\) be an \(S(2, 3, \ccv)\).
By the Veblen-Young theorem \cite{veblen-projective-1916}, the following two conditions are equivalent:
\begin{enumerate}
\item \(\cD\) is projective, that is, it is isomorphic to a projective Steiner triple system.
\item If \(\set{a, b, c}\), \(\set{a, d, e}\), and \(\set{b, d, f}\) are distinct three blocks of \(\cD\), then \(\set{c, e, f}\) is also a block (see Figure \ref{fig:org66c4986}).
\end{enumerate}
\begin{figure}[H]
\centering
\includegraphics[scale=0.52]{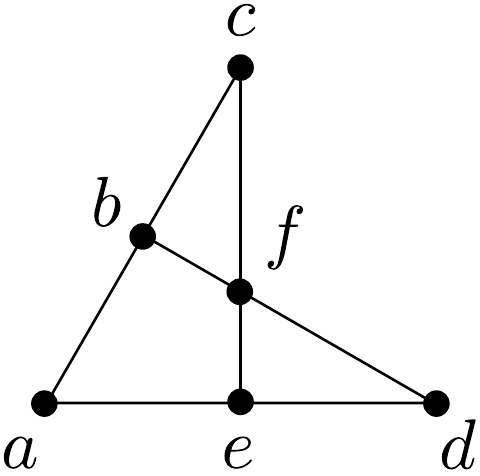}
\caption{\label{fig:org66c4986}A characterization of projective Steiner triple systems.}
\end{figure}
 
\end{remark}

\subsection{Impartial Games}
\label{sec:org5883af2}
\label{org7251d2b}
\mylink{sec-games}We recall the concept of \textbf{an} impartial game (see, for example, \cite{berlekamp-Winning-2001, conway-numbers-2001} for \textbf{more} details).

\yon
We represent a game as a digraph.
Let \(\Gamma\) be a digraph \((\Position{\Gamma}, \Edge{\Gamma})\) satisfying
\begin{equation}
\label{equ:def-game}
 \lg_{\Gamma}(\ccP) < \infty \quad \text{for every\ } \ccP \in \Position{\Gamma},
\end{equation}
\yoff
where \(\lg_{\Gamma}(\ccP)\) is the maximum length of a walk from \(\ccP\).
The vertex set \(\Position{\Gamma}\) is called the \emph{position set} of \(\Gamma\). 
If \((\ccP, \ccQ) \in \Edge{\Gamma}\), then \(\ccQ\) is called an \emph{option} of \(\ccP\).
A position without options is called a \textbf{\emph{terminal position}}.
\yon
Suppose that \(\Gamma\) has at least one position; then \(\Gamma\) has a terminal position.
We consider \(\Gamma\) as a two-player game in the way described in the introduction.
This kind of game is called a (short) impartial game, and every impartial game can be
represented in this manner.
Thus we call a digraph satisfying \eqref{equ:def-game} a \emph{game graph} or simply a \emph{game}.\footnote{\textbf{Although a (short) impartial game is defined as a vertex in a game graph \(\Gamma\), we can consider \(\Gamma\) itself as an impartial game by adding a vertex if necessary (see Remark \ref{org71be464}). Therefore we call \(\Gamma\) simply a game.}}
\yoff

 \begin{example}
 \comment{Exm.}
\label{sec:orgdc4ff6c}
\label{orgb85effe}
Let 
\[
 \Gamma^1 = \bigl([4], \bigfamily{(3, 2), (3, 1), (2, 1), (2, 0), (1, 0)}\bigr)\ \tand\  \Gamma^2 = \bigl([2], \bigfamily{(1, 0), (0, 1)}\bigr).
\]
See Figure \ref{fig:orgec3afb4}.
Then \(\Gamma^1\) is a game and position 0 is a \textbf{terminal} position in \(\Gamma^1\).
However, \(\Gamma^2\) is not a game since the walk \((1, 0, 1, 0, \dotsc)\) has infinite length.

\begin{figure}[htbp]
\centering
\includegraphics[scale=0.52]{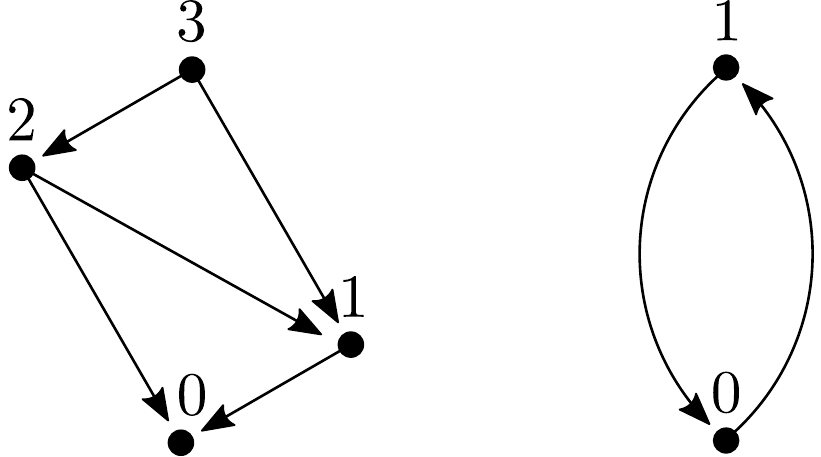}
\caption{\label{fig:orgec3afb4}\(\Gamma^1\) and \(\Gamma^2\).}
\end{figure}
 
\end{example}

 \begin{example}[Welter's game and the hexad game]
 \comment{Exm. [Welter's game and the hexad game]}
\label{sec:org2bbf4fc}
\label{orgb214729}
Let \(\Wel{\ccv}{\cck}\) be the game with position set \({[\ccv] \choose \cck}\) and
edge set
\[
\biggset{(\ccP, \ccP^{(\ccp\ \ccq)}) \in {[\ccv] \choose \cck}^2 : \ccq < \ccp,\ \ccp \in \ccP,\ \ccq \not \in \ccP}.
\]
The game \(\Wel{\ccv}{\cck}\) is called \emph{Welter's game}.
Let
\[
 \Position{\Hexad} = \biggset{\ccP \in {[12] \choose 6} : \sum_{\ccp \in \ccP} \ccp \ge 21}.
\]
\textbf{The subgraph \(\Induced{\Wel{12}{6}}{\Position{\Hexad}}[\bigcbrackets]\) induced in \(\Wel{12}{6}\) by \(\Position{\Hexad}\) is called the \emph{hexad game} and denoted by \(\Hexad\).}
 
\end{example}

\comment{Outcome}
\label{sec:orgaae06db}
We next define \textbf{\(\sP\)-position} sets.
Let \(\sP\) and \(\sN\) be two symbols.
For a game \(\Gamma\) and a position \(\ccP \in \Position{\Gamma}\), 
the \emph{outcome} of \(\ccP\) is defined recursively by
\[
 \cco_{\Gamma}(\ccP) = \begin{cases}
 \sN & \tif \ccP \text{\ has an option\ } \ccQ \text{\ with\ } \cco_{\Gamma}(\ccQ) = \sP, \\
 \sP & \totherwise. \mylink{p6-1}
 \end{cases}
\]
Recall that \(\Gamma\) has a \textbf{terminal} position if it has at least one position.
By definition, the outcome of a \textbf{terminal} position is \(\sP\).
\textbf{An easy induction shows} that the previous player can force a win if and only if the outcome of the starting position is \(\sP\).
Let 
\[
 \Winning{\Gamma} = \bigset{\ccP \in \Position{\Gamma} : \cco_{\Gamma}(\ccP) = \sP}.
\]
The set \(\Winning{\Gamma}\) is called the \textbf{\emph{\(\sP\)-position set} or} \emph{winning position set} of the game \(\Gamma\).
\yon
\mylink{p6-18}Note that \(\Winning{\Gamma}\) is an independent set in \(\Gamma\),
that is, \((\ccB, \ccC) \not \in \Edge{\Gamma}\) for any \(\ccB, \ccC \in \Winning{\Gamma}\).
As we will see in the next section,
our construction, described in the introduction, can be applied to an independent set in \(\Gamma\).
Note that if \(\Block\) is the block set of an \(\ccS(\cct, \cck, \ccv)\) with \(\cct < \cck\),
then \(\Block\) is an independent set in \(\Wel{\ccv}{\cck}\).
\yoff

 \begin{example}
 \comment{Exm.}
\label{sec:orgb70e643}
\label{org8bc2902}
Let \(\Gamma^1\) be as in Example \ref{orgb85effe}.
The outcomes of the four positions are shown in Figure~\ref{fig:org6a95754}.
Note that the previous player can force a win if we start at position 3.
\begin{figure}[H]
\centering
\includegraphics[scale=0.52]{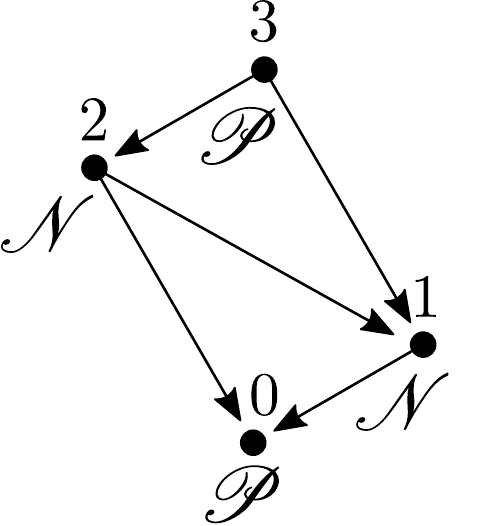}
\caption{\label{fig:org6a95754}Outcomes of positions.}
\end{figure}
 
\end{example}

\begin{remark}
 \comment{Rem.}
\label{sec:org1e5f873}
\label{orge3d702b}
\yon
\mylink{p7-1}Let \(\shD = ([12], \Winning{\Hexad})\). Then
\(\shD\) is an \(\ccS(5, 6, 12)\), called the Steiner system \(\ccS(5, 6, 12)\) in the \emph{shuffle numbering} \cite{conway-lexicographic-1986, kahane-hexad-2001}.
For example, if \(\ccP\) is an element in \({[12] \choose 6}\) with \(\sum_{\ccp \in \ccP} \ccp = 21\),
then it is a terminal position in \(\Hexad\), so \(\ccP \in \Winning{\Hexad}\).
Incidentally, the whole \(\sP\)-positions of \(\Hexad\) can be obtained as the orbit
\[
 \bigset{\set{0, 1, 2, 3, 4, 11}^\pi : \pi \in \ccG},
\]
where \(\set{0, 1, 2, 3, 4, 11}^\pi = \set{0^\pi, 1^\pi, 2^\pi, 3^\pi, 4^\pi, 11^\pi}\) and \(\ccG\) is the group generated by the two permutations \(\cci \mapsto 11 - \cci\) and \(\cci \mapsto \min(2 \cci, 23 - 2 \cci)\) of \([12]\),
which is the automorphism group of \(\shD\).
We can also explicitly describe \(\Winning{\Hexad}\) by using the
following \(3 \times 4\) array, called the MINIMOG in the shuffle numbering (see \cite{conway-sphere-2013}).
\begin{center}
\begin{center}
\includegraphics[width=3cm]{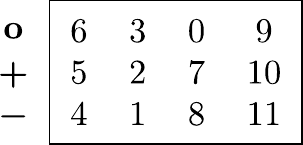}
\label{orgca21130}
\end{center}
\end{center}
\yoff
 
\end{remark}

\section{Construction of Games}
\label{sec:org08f32ea}
\label{org887ee75}
\yon
For an independent set \(\Block\) in a game \(\Gamma\), we construct a game whose \(\sP\)-position set equals \(\Block\).
In particular, our construction works for the block set \(\Block\) of an \(\ccS(\cct, \cck, \ccv)\) with \(\cct < \cck\) since \(\Block\) is an independent set in \(\Wel{\ccv}{\cck}\). 
\yoff

\textbf{For an independent set \(\Block\) in a game \(\Gamma\), define \(\bPosition{\Block}[\Gamma]\) by \eqref{equ:def-bposition}.}
When no confusion can arise, we write \(\bPosition{\Block}\), \(\InNeighbors{\ccP}\), and \(\OutNeighbors{\ccP}\) instead of \(\bPosition{\Block}[\Gamma]\), \(\InNeighbors{\ccP}[\Gamma]\), and \(\OutNeighbors{\ccP}[\Gamma]\), respectively.
\textbf{Let \(\Phi(\Gamma, \Block)\) denote the family of the position sets of induced subgraphs of \(\Gamma\) whose \(\sP\)-position sets equal \(\Block\), that is,} \yon
\[
 \ymodoff{\Phi(\Gamma, \Block) = \bigset{\cQ \subseteq \Position{\Gamma} : \Winning{\Induced{\Gamma}{\cQ}} = \Block}.}
\]
 
 \begin{theorem}
 \comment{Thm.}
\label{sec:org7adaf8e}
\label{orgeab1e99} 
\textbf{If \(\Gamma\) is a game and \(\Block\) is an independent set in \(\Gamma\), then} \yoff
\[
 \Phi(\Gamma, \Block) = \bigset{\cQ : \Block \subseteq \cQ \subseteq \bPosition{\Block}}\ymod{;}
\]
\ymod{i}n particular, \(\Winning{\Induced{\Gamma}{\bPosition{\Block}}} = \Block\) and \(\bPosition{\Block}\) is the maximum element of \(\Phi(\Gamma, \Block)\) with respect to inclusion.
 
\end{theorem}

\comment{Connect}
\label{sec:orgf67df3b}
\yoff
\begin{proof}
 \comment{Proof.}
\label{sec:org73f7fb6}
We begin by showing that if \(\cQ\) is a subset of \(\Position{\Gamma}\) with \(\Block \subseteq \cQ \subseteq \bPosition{\Block}\), then \(\cQ \in \Phi(\Gamma, \Block)\).
Let \(\Delta = \Induced{\Gamma}{\cQ}\). It suffices to show that \(\Winning{\Delta} = \Block\). Let \(\ccQ \in \cQ\) and \(\ccl = \lg_{\Delta}(\ccQ)\).
We show that \(\ccQ \in \Winning{\Delta}\) if and only \(\ccQ \in \Block\) by induction on \(\ccl\).
If \(\ccl = 0\), then \(\ccQ \in \Winning{\Delta}\) and \(\ccQ \in \Block\) since \(\ccQ\) is a \textbf{terminal} position of \(\Delta\) and \(\ccQ \in \cQ \subseteq \bPosition{\Block}\).
Suppose that \(\ccl > 0\).
First, we show that if \(\ccQ \in \Block\), then \(\ccQ \in \Winning{\Delta}\). Let \(\ccR\) be an option of \(\ccQ\).
Since \(\Block\) is \textbf{an independent set}, it follows that \(\ccR \not \in \Block\). 
By the induction hypothesis, \(\ccR \not \in \Winning{\Delta}\).
This implies that \(\ccQ \in \Winning{\Delta}\).
Next, we show that if \(\ccQ \not \in \Block\), then \(\ccQ \not \in \Winning{\Delta}\).
Since \(\ccQ \in \cQ \subseteq \bPosition{\Block}\), it follows that \(\ccQ\) has an option \(\ccB\) with \(\ccB \in \Block\).
By the induction hypothesis, \(\ccB \in \Winning{\Delta}\), so \(\ccQ \not \in \Winning{\Delta}\).
Therefore \(\Winning{\Delta} = \Block\).

It remains to prove that if \(\cQ \in \Phi(\Gamma, \Block)\), then \(\Block \subseteq \cQ \subseteq \bPosition{\Block}\).
It is obvious that \(\cQ \supseteq \Block\). We show that \(\cQ \subseteq \bPosition{\Block}\).
Let \(\ccQ \in \cQ\) and \(\Delta = \Induced{\Gamma}{\cQ}\).
If \(\ccQ \in \Block\), then \(\ccQ \in \bPosition{\Block}\).
Suppose that \(\ccQ \not \in \Block\).
Since \(\Block = \Winning{\Delta}\), it follows that \(\ccQ \not \in \Winning{\Delta}\),
and that \(\ccQ\) has an option \(\ccB\) with \(\ccB \in \Winning{\Delta} = \Block\).
This implies that \(\ccQ \in \InNeighbors{\ccB} \subseteq \bPosition{\Block}\).
Therefore \(\cQ \subseteq \bPosition{\Block}\).
\end{proof}

\comment{Connect.}
\label{sec:org78111e5}
For a \(\design{\cct}{\ccv, \cck, \lambda}\) \(\cD\),
let \(\DWel{\cD}\) denote the game \(\Induced{\Wel{\ccv}{\cck}}{\bPosition{\cD}}[\bigcbrackets]\), where \(\bPosition{\cD} = \bPosition{\Block[\cD]}\).
\yon
If \(\Block[\cD]\) is an \textbf{independent set} in \(\Wel{\ccv}{\cck}\),
it follows from Theorem \ref{orgeab1e99} that \(\Winning{\DWel{\cD}} = \Block[\cD]\).
In particular, we obtain the following corollary.
\yoff

 \begin{corollary}
 \comment{Cor.}
\label{sec:orgb5ff674}
\label{orgc92db7a}
If \(\cD\) is an \(S(\cct, \cck, \ccv)\) with \(\cct < \cck\), then 
the \textbf{\(\sP\)-position} set of \(\DWel{\cD}\) equals \(\Block[\cD]\).
 
\end{corollary}

 \begin{example}
 \comment{Exm.}
\label{sec:orgc644fdc}
\label{org39cf05f}
\textbf{Let \(\cD\) be the \(\ccS(1, 2, 4)\) with block set \(\bigfamily{\set{0, 2}, \set{1, 3}}\).}
\textbf{As we have seen in the introduction, \(\bPosition{\cD} = {[4] \choose 2} \setminus \bigfamily{\set{0, 1}}\) and \(\Winning{\DWel{\cD}} = \bigfamily{\set{0, 2}, \set{1, 3}}\).}
 
\end{example}

 \begin{example}
 \comment{Exm.}
\label{sec:orgbec1df9}
\label{org7d0242b}
\yon
The game \(\DWel{\shD}\) equals \(\Hexad\).
To verify this, it suffices to show that \(\bPosition{\shD} = \Position{\Hexad}\).
Since \(\Winning{\Hexad}\) equals the block set \(\shBlock\) of \(\shD\),
it follows from Theorem \ref{orgeab1e99} that \(\Position{\Hexad} \subseteq \bPosition{\shD}\).
\yoff
We show that \(\bPosition{\ymod{\shD}} \subseteq \Position{\Hexad}\).
If \(B \in \shBlock\) \ymod{(}$= \Winning{\Hexad}$\ymod{)}, then \(B \in \Position{\Hexad}\), so \(\displaystyle \sum_{b \in B} b \ge 21\).
This implies that if \(\ccP \in \bPosition{\ymod{\shD}}\), then \(\displaystyle \sum_{p \in \ccP} p \ge 21\).
Thus \(\bPosition{\ymod{\shD}} \subseteq \Position{\Hexad}\). Therefore \(\bPosition{\ymod{\shD}} = \Position{\Hexad}\).
 
\end{example}

\begin{remark}
 \comment{Rem.}
\label{sec:org1ca857f}
\label{org94c567d}
\mylink{p9-22}\mylink{p8-f}
In this paper, we will investigate \textbf{an} \(S(t, t + 1, v)\) by using the above games obtained from Welter's game.
Incidentally, we can also consider games obtained from other than Welter's game to study \(S(t, t + m, v)\), especially when \(\ccm > 1\).\footnote{\textbf{If \(\cD\) is an \(\ccS(\cct, \cct + \ccm, \ccv)\) with \(\ccm > 1\), then we can show that the number of positions in \(\DWel{\cD}\) always equals \(\lambda_0 + \bigl(\lambda_0\cck(\ccv - \cck)/2\bigr)\) by Lemma \ref{org0f0d991}. Thus to investigate \(\cD\) by using its game distribution, we have to consider other games, for example, \(\Induced{\Gamma}{\bPosition{\Block}[\Gamma]}[\bigcbrackets]\) described in Remark \ref{org94c567d}.}}
For example, let \(\cck = \cct + \ccm\) and let \(\Wel{\ccv}{\cck}^m\) be the game with position set \({[v] \choose k}\) and edge set
\[
 \biggset{(P, Q) \in {[v] \choose k}^2 : P \neq Q, \text{\ there is a walk from\ } P \text{\ to } Q \text{\ of length at most\ } m \text{\ in\ } \Wel{\ccv}{\cck} }.
\]
For example, \(\bigl(\set{1, 4, 5}, \set{0, 2, 5}\bigr)\) is an edge in \(\Gamma_{6, 3}^2\) since \(\bigl(\set{1, 4, 5}, \set{0, 4, 5}, \set{0, 2, 5}\bigr)\) is a walk of length 2 in \(\Wel{6}{3}\).
Note that \(\Wel{\ccv}{\cck}^1\) is Welter's game \(\Wel{\ccv}{\cck}\). Let \(\Gamma = \Wel{\ccv}{\cck}^m\) and let \(\cD\) be an \(S(t, \cck, \ccv)\) with block set \(\Block\). 
Then we can see that \(\Block\) is \textbf{an independent set} in \(\Gamma\). 
Therefore the \textbf{\(\sP\)-position} set of \(\Induced{\Gamma}{\bPosition{\Block}[\Gamma]}[\bigcbrackets]\) is equal to \(\Block\).
 
\end{remark}

\comment{connect}
\label{sec:org6e73aaa}
\yon
\begin{remark}
 \comment{Rem.}
\label{sec:orgf59e17a}
\label{org71be464}
\mylink{p8-r14}As we have mentioned, a game graph itself can be considered as a (short or long) impartial game 
by adding a new vertex described below if necessary; our construction also works after adding the new vertex.
Let \(\Gamma\) be a game graph with at least one vertex.
If \(\ccP\) and \(\ccQ\) are two vertices in \(\Gamma\), then \(\ccQ\) is called a \emph{descendant} of \(\ccP\) if there is a path from \(\ccP\) to \(\ccQ\).
A \emph{short impartial game} can be defined as a vertex \(\ccP\) in \(\Gamma\) and a \emph{position} in \(\ccP\) as a descendant of \(\ccP\).
Thus if \(\Gamma\) has a vertex \(\ccP\) such that the set of descendants of \(\ccP\) equals \(\Position{\Gamma}\), 
then we can identify \(\Gamma\) with \(\ccP\), so \(\Gamma\) can be considered as a short impartial game.
For example, if \(\Gamma\) is Welter's game \(\Wel{\ccv}{\cck}\), then
the vertex \(\set{\ccv - \cck, \ccv - \cck + 1, \dotsc, \ccv - 1}\) satisfies the above condition.
Even when \(\Gamma\) does not have such a vertex,
we can add a new vertex satisfying this condition to \(\Gamma\) while preserving its \(\sP\)-positions as follows.
Let \(\ccP_{\infty}\) be a new vertex and let \(\tilde{\Gamma}\) be a digraph
with vertex set \(\Position{\Gamma} \cup \set{\ccP_\infty}\) and edge set
\[
 \Edge{\Gamma} \cup \bigset{(\ccP_\infty, \ccP) : \ccP \in \Position{\Gamma}}.
\]
Then we can identify \(\tilde{\Gamma}\) with \(\ccP_\infty\).
Note that if \(\bigset{\lg_{\Gamma}(\ccP) : \ccP \in \Position{\Gamma}} = \NN\),
then \(\lg_{\tilde{\Gamma}}(\ccP_\infty)\) is not finite, which equals the first infinite ordinal \(\omega\), 
and hence \(\tilde{\Gamma}\) does not satisfy \eqref{equ:def-game}.
Such \(\tilde{\Gamma}\) is called a long impartial game.
Even if \(\tilde{\Gamma}\) is long, our construction works also when we use \(\tilde{\Gamma}\) instead of \(\Gamma\) because
\(\ccP_\infty\) is an \(\sN\)-position and \(\Winning{\tilde{\Gamma}} = \Winning{\Gamma}\).
 
\end{remark}

\comment{connect}
\label{sec:org91f3137}
\yoff

\section{Game Distributions}
\label{sec:org8a17fb3}
\label{org927a8a4}

We introduce the game distribution of a Steiner system in this section.
The Steiner system \(\shD\) will be characterized by using its game distribution.
Hereafter, for \(\ccP \in {[\ccv] \choose \cck}\), let
\[
 \InNeighbors{\ccP} = \InNeighbors{\ccP}[\Wel{\ccv}{\cck}] \tand \OutNeighbors{\ccP} = \OutNeighbors{\ccP}[\Wel{\ccv}{\cck}].
\]

For a \(\design{\cct}{\ccv, \cck, \lambda}\ \cD\) and a permutation \(\pi\) in \(\Sym([\ccv])\), let \(\cD^\pi\) denote the design obtained from \(\cD\) by applying \(\pi\), that is, \(\cD^\pi = \bigl([\ccv], \Block[\cD]^\pi\bigr)\), where 
\(\Block[\cD]^\pi = \bigset{\ccB^\pi : \ccB \in \Block[\cD]}\) and \(\ccB^\pi = \bigset{\ccb^\pi : \ccb \in \ccB}\).
Define
\[
 \Orbit{\cD} = \bigset{\cD^\pi  : \pi \in \Sym([\ccv])}.
\]
For example, if \(\cD = \bigl([4], \bigfamily{\set{0, 2}, \set{1, 3}}\bigr)\), then
\begin{align*}
 \Orbit{\cD} &= \bigset{\cD, \cD^{(1\ 2)}, \cD^{(2\ 3)}} \\
 &= \Bigset{\bigl([4], \bigfamily{\set{0, 2}, \set{1, 3}}\bigr),\ \bigl([4], \bigfamily{\set{0, 1}, \set{2, 3}}\bigr),\ \bigl([4], \bigfamily{\set{0, 3}, \set{1, 2}}\bigr)}.
\end{align*}
Note that
\[
 \size{\Orbit{\cD}} = \frac{\size{\Sym([\ccv])}}{\size{\Aut(\cD)}} = \frac{\ccv!}{\size{\Aut(\cD)}}.
\]
Let \(\Freq{\cD}(\ccn)\) denote the number of designs \(\cD' \in \Orbit{\cD}\) such that
the corresponding game \(\DWel{\cD'}\) has \(\ccn\) positions, that is,
\[
 \Freq{\cD}(\ccn) = \size{\Set{\cD' \in \Orbit{\cD} : \msize{\bPosition{\cD'}} = \ccn}}.
\]
The function \(\Freq{\cD}\) will be called the \emph{game distribution} of \(\cD\) \textbf{(\emph{with respect to \(\Wel{\ccv}{\cck}\)})}.

 \begin{example}
 \comment{Exm.}
\label{sec:orga16fd6f}
\label{orgc7f5bd0}
Let \(\cD\) be \textbf{an} \(\ccS(1, 2, 4)\). 
\textbf{As we have seen in the introduction, the game distribution of \(\cD\)} is as follows:
\begin{center}
\begin{tabular}{c|ccc|c}
\(\ccn\) & 4 & 5 & 6 & Total\\
\hline
\(\Freq{\cD}(\ccn)\) & 1 & 1 & 1 & 3\\
\end{tabular}
\end{center}
 
\end{example}

 \begin{example}
 \comment{Exm.}
\label{sec:org934f2c7}
\label{orgd9071c5}
Let \(\cD\) be an \(S(2, 3, 7)\).
Then the game distribution of \(\cD\) is as follows:
\mylink{p10-20}
\begin{center}
\begin{tabular}{c|cccccccc|c}
\(\ccn\) & 28 & 29 & 30 & 31 & 32 & 33 & 34 & 35 & Total\\
\hline
\(\Freq{\cD}(\ccn)\) & 1 & 3 & 5 & 6 & 6 & 5 & 3 & 1 & 30\\
\end{tabular}
\end{center}
\mylink{p10-21}
Note that this distribution is symmetric, that is, \(\Freq{\cD}(\ccn) = \Freq{\cD}(63 - \ccn)\). 
Incidentally, the game distribution of a Steiner system is not always symmetric. 
For example, if \(\cD\) is an \(S(2, 3, 9)\), then its game distribution is as follows:
{\tabcolsep = 1mm
\begin{center}
\begin{tabular}{c|ccccccccccccc|c}
\(\ccn\) & 68 & 69 & 70 & 71 & 72 & 73 & 74 & 75 & 76 & 77 & 78 & 79 & 80 & Total\\
\hline
\(\Freq{\cD}(\ccn)\) & 1 & 6 & 16 & 36 & 77 & 94 & 116 & 129 & 131 & 104 & 74 & 39 & 17 & 840\\
\end{tabular}
\end{center}
}
 
\end{example}

 \begin{theorem}
 \comment{Thm.}
\label{sec:org8ff3979}
\label{org0f36e55}
Let \(\cD\) be an \(S(5, 6, 12)\).
The game distribution of \(\cD\) is as follows:
{\tabcolsep = 1mm
\begin{center}
\begin{tabular}{c|cccccccccccc|c}
\(\ccn\) & \(905\) & \(906\) & \(907\) & \(908\) & \(909\) & \(910\) & \(911\) & \(912\) & \(913\) & \(914\) & \(915\) & \(916\) & \(\mytext{Total}\)\\
\hline
\(\Freq{\cD}(\ccn)\) & \(1\) & \(10\) & \(42\) & \(150\) & \(351\) & \(650\) & \(1012\) & \(1237\) & \(939\) & \(532\) & \(115\) & \(1\) & \(5040\)\\
\end{tabular}
\end{center}
}
\noindent
Moreover, \(\DWel{\shD}\) is the hexad game and is the unique minimum game, that is, \(\size{\bPosition{\shD}} = 905\).
 
\end{theorem}

\comment{Connect}
\label{sec:orgfd4981b}
Theorem \ref{org0f36e55} can be obtained by computer.

\begin{remark}
 \comment{Rem.}
\label{sec:org1c8d726}
Let \(\cD\) be the unique \(S(5, 6, 12)\) with \(\size{\bPosition{\cD}} = 916\).
Then 
\begin{align*}
 {[12] \choose 6} \setminus \bPosition{\cD} = \{&\set{0, 1, 2, 3, 8, 9}, \set{0, 1, 2, 4, 7, 9}, \set{0, 1, 2, 4, 7, 10}, \set{0, 1, 2, 5, 7, 8}, \\
 &\set{0, 1, 3, 5, 6, 8}, \set{0, 1, 3, 5, 7, 9}, \set{0, 1, 4, 5, 6, 7}, \set{0, 2, 3, 5, 6, 8}\}
\end{align*}
and the block set \(\Block[\cD]\) can be obtained from the following MINIMOG:
\begin{center}
\begin{center}
\includegraphics[width=3cm]{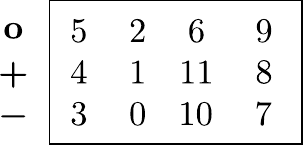}
\end{center}
\end{center}
 
\end{remark}

\comment{Connect}
\label{sec:org449476b}
\yon
\begin{remark}
 \comment{Rem.}
\label{sec:orgc1d5ea7}
\label{orgb6d5185}
We close this section by remarking that the game distribution of \(\ccS(1, 2, 2 \ccv)\) can be calculated as follows.
\yoff
If \(\cD\) is an \(S(1, 2, 2v)\), then
\begin{equation}
\label{equ:exm:1-2-2v-formula}
  \sum_{n = 0}^\infty  \Freq{\cD}(\ccn) \ccx^{\ccn} = \ccx^{v^2} \prod_{i = 1}^{\ccv} \frac{\ccx^{2i - 1} - 1}{\ccx - 1}.
\end{equation}
\textbf{For example, if \(\cD\) is an \(\ccS(1, 2, 6)\), then its game distribution and \(\sum \Freq{\cD}(\ccn) \ccx^\ccn\) are as follows:}
\begin{center}
\begin{tabular}{c|ccccccc|c}
\(\ccn\) & 9 & 10 & 11 & 12 & 13 & 14 & 15 & Total\\
\hline
\(\Freq{\cD}(\ccn)\) & 1 & 2 & 3 & 3 & 3 & 2 & 1 & 15\\
\end{tabular}
\end{center}
\[
 \sum_{\ccn = 0}^\infty \Freq{\cD}(\ccn) \ccx^n = \ccx^9 (\ccx^2 + \ccx + 1)(\ccx^4 + \ccx^3 + \ccx^2 + \ccx + 1).
\]
\textbf{We sketch the proof.}
Let \(\Freq~{\cD}(\ccn)\) denote the number of designs \(\cD' \in \Orbit{\cD}\) such that
the corresponding game \(\DWel{\cD'}\) has \(\ccn\) non-positions, that is,
\[
 \Freq~{\cD}(\ccn) = \size{\Set{\cD' \in \Orbit{\cD} : \msize{\bPosition~{\cD'}} = \ccn}},
\]
\textbf{where \(\bPosition~{\cD} = {[2 \ccv] \choose 2} \setminus \bPosition{\cD}\)}.
Since \(\bigsize{\bPosition~{\cD}} + \bigsize{\bPosition{\cD}} = \binom{2v}{2} = \ccv(2\ccv - 1)\), \textbf{it follows} that \(\Freq~{\cD}(\ccn) = \Freq{\cD}(\ccv(2 \ccv - 1) - \ccn)\).
We \textbf{first} show that
\begin{equation}
\label{equ:exm:1-2-2v}
 \ymod{\sum_{n = 0}^\infty  \Freq~{\cD}(\ccn) \ccx^{\ccn}}  =  \prod_{i = 1}^{\ccv} \frac{\ccx^{2i - 1} - 1}{\ccx - 1}.
\end{equation}
\textbf{Let \(\Psi_{\ccv}(\ccx)\) be the left-hand side of \eqref{equ:exm:1-2-2v}.}
\textbf{We show \eqref{equ:exm:1-2-2v}} by induction on \(\ccv\). 
If \(\ccv = 1\), then \(\Freq~{\cD}(0) = 1\) and \(\Freq~{\cD}(\ccn) = 0\) for \(\ccn \ge 1\), so \eqref{equ:exm:1-2-2v} holds.
\yon
Suppose that \(\ccv \ge 2\), and let \(\set{0, \ccw} \in \Block[\cD]\).
We calculate \(\msize{\bPosition~{\cD}}\).
The number of \(\ccP \in \bPosition~{\cD}\) with \(\ccP \cap \set{0, \ccw} \neq \emptyset\)
equals \(\ccw - 1\) since the set of such \(\ccP\) is \(\bigfamily{\set{0, 1}, \dotsc, \set{0, \ccw - 1}}\).
Moreover, the number of \(\ccP \in \bPosition~{\cD}\) with \(\ccP \cap \set{0, \ccw} = \emptyset\) (i.e., \(\ccP \subseteq [2 \ccv] \setminus \set{0, \ccw}\))
equals \(\msize{\bPosition~{\cE}}\), where \(\cE\) is the \(\ccS(1, 2, 2 \ccv - 2)\) described below.
Write \([2 \ccv] \setminus \set{0, \ccw} = \set{\cca_0, \dotsc, \cca_{2 \ccv - 3}}\) with \(\cca_0 < \dotsb < \cca_{2 \ccv - 3}\),
and let \(\phi\) be the bijection \([2v] \setminus \set{0, \ccw} \ni \cca_\cci \mapsto \cci \in [2 \ccv - 2]\). 
Note that \(\phi\) is order-preserving, that is, \(\phi(\cca) < \phi(\ccb)\) if and only if \(\cca < \ccb\).
Let \(\cE = \bigl([2 \ccv - 2], \bigset{\phi(\ccB) : \ccB \in \Block[\cD], \ccB \neq \set{0, \ccw}}\bigr)\).
Then \(\cE\) is an \(\ccS(1, 2, 2 \ccv - 2)\).
If \(\ccQ \subseteq [2 \ccv] \setminus \set{0, \ccw}\),
then \(\ccQ \in \bPosition~{\cD}\) if and only if \(\phi(\ccQ) \in \bPosition~{\cE}\)  since \(\phi\) is order-preserving.
Therefore the number of \(\ccP \in \bPosition~{\cD}\) with \(\ccP \subseteq [2 \ccv] \setminus \set{0, \ccw}\)
is equal to \(\msize{\bPosition~{\cE}}\). Thus \(\msize{\bPosition~{\cD}} = \ccw - 1 + \msize{\bPosition~{\cE}}\).
\yoff
It follows from the induction hypothesis that
\[
 \Psi_{\ccv}(\ccx) = (1 + \ccx + \dotsb + \ccx^{2\ccv - 2}) \Psi_{\ccv - 1}(\ccx) = \prod_{i = 1}^{\ccv}\frac{\ccx^{2i - 1} - 1}{\ccx - 1}.
\]
Therefore \eqref{equ:exm:1-2-2v} holds.
\textbf{Since \(\Freq{\cD}(\ccn) = 0\) for \(\ccn > \ccv(2 \ccv - 1)\) and} \(\Freq{\cD}(\ccn) = \Freq~{\cD}(\ccv(2\ccv - 1) - \ccn)\),
we see that
\begin{align*}
 \sum_{n = 0}^{\infty}  \Freq{\cD}(\ccn) \ccx^{\ccn} &= \sum_{n = 0}^{\ccv(2 \ccv - 1)}  \Freq~{\cD}(\ccv (2 \ccv - 1) - \ccn) \ccx^{\ccn}  = \sum_{n = 0}^{\ccv(2 \ccv - 1)}  \Freq~{\cD}(\ccn) \ccx^{\ccv (2 \ccv - 1) - \ccn}  \\
 &= \ccx^{\ccv(2 \ccv - 1)} \prod_{i = 1}^{\ccv} \ymod{(1 + x^{-1} + \dotsb + x^{-2i + 2})} = \ccx^{\ccv^2} \prod_{i = 1}^{\ccv}\frac{\ccx^{2i - 1} - 1}{\ccx - 1}.
\end{align*}
 
\end{remark}

\comment{connect}
\label{sec:orgab691cc}
\yoff

\section{Symmetric Components}
\label{sec:org77e62c4}
\label{orgee4c4c8}
\mylink{p13-6}
We show that the game distribution of an \(S(\cct, \cct + 1, \ccv)\) can be decomposed into \emph{symmetric components}. 
Using symmetric components, we characterize projective Steiner triple systems in Section \ref{orge91698d}.

For a \(\design{\cct}{\ccv, \cck, \lambda}\) \(\cD\) and a non-block \(\ccP \in {[\ccv] \choose k} \setminus \Block[\cD]\), a block \(\ccB \in \Block[\cD]\) is called an \emph{out-block} of \(\ccP\) if \(\ccB \in \OutNeighbors{\ccP}\).
Define \(\cca_i(\cD)\) to be the number of non-blocks with \(i\) out-blocks, that is,
\[
 \cca_i(\cD) = \size{\Set{\ccP \in {[\ccv] \choose \cck} \setminus \Block[\cD]: \bigsize{\OutNeighbors{\ccP} \cap \Block[\cD]} = i}}.
\]
For example, if \(\cD = \bigl([4], \bigfamily{\set{0, 2}, \set{1, 3}}\bigr)\), then \(\set{0, 1}\) has no out-blocks, \(\set{0, 3}\) and \(\set{1, 2}\) has an out-block \(\set{0, 2}\), and \(\set{2, 3}\) has two out-blocks \(\set{0, 2}\) and \(\set{1, 3}\), so
\[
\cca_0(\cD) = \size{\bigfamily{\set{0, 1}}} = 1,\ \cca_1(\cD) = \size{\bigfamily{\set{0, 3}, \set{1, 2}}} = 2,\ \cca_2(\cD) = \size{\bigfamily{\set{2, 3}}} = 1.
\]
Note that \(\size{\bPosition{\cD}} = \size{\Block[\cD]} + \sum_{i \ge 1} \cca_i(\cD)\) and \(\msize{\bPosition~{\cD}} = \cca_0(\cD)\),
\textbf{where \(\bPosition~{\cD} = {[\ccv] \choose \cck} \setminus \bPosition{\cD}\)}.
\yon
Moreover, if \(\Block[\cD]\) is an independent set in \(\Wel{\ccv}{\cck}\), then \(\cca_{\cci}(\cD) = 0\) for \(\cci \ge \cck + 1\).
Indeed, if \(\size{\OutNeighbors{\ccP} \cap \Block[\cD]} \ge \cck + 1\) for some \(\ccP \in {[\ccv] \choose \cck} \setminus \Block[\cD]\),
then \(\ccP^{(\ccp\ \ccq)}, \ccP^{(\ccp\ \ccq')} \in \Block[\cD]\) for some \(\ccp \in \ccP\) and some \(\ccq, \ccq' \not \in \ccP\) with \(\ccq' < \ccq\), 
so \(\Block[\cD]\) is not an independent set because \((\ccP^{(\ccp\ \ccq)}, \ccP^{(\ccp\ \ccq')})\) is an edge of \(\Wel{\ccv}{\cck}\).
\yoff
Let
\[
 \ccs(\cD) = \Set{\cca_0(\cD') + \cca_{\cck}(\cD') : \cD' \in \Orbit{\cD}}
\]
and
\[
 \Orbit{\cD}[\alpha] = \Set{\cD' \in \Orbit{\cD} : \cca_0(\cD') + \cca_{\cck}(\cD') = \alpha}
\]
for \(\alpha \in \ccs(\cD)\).
Define
\[
 \Freq{\cD}(\ccn)[\alpha] = \size{\Set{\cD' \in \Orbit{\cD}[\alpha] :  \msize{\bPosition{\cD'}} = \ccn}}
\]
and
\begin{align*}
 \Freq~{\cD}(\ccn)[\alpha] &= \size{\Set{\cD' \in \Orbit{\cD}[\alpha] :  \msize{\bPosition~{\cD'}} = \ccn}} \\
 &=  \bigsize{\bigset{\cD' \in \Orbit{\cD}[\alpha] :  \cca_0(\cD') = \ccn}}. \mylink{p13-25}
\end{align*}
By definition,
\[
 \Freq{\cD}(\ccn) = \sum_{\alpha \in s(\cD)} \Freq{\cD}(\ccn)[\alpha].
\]

Let \(\cD\) be an \(S(\cct, \cct + 1, \ccv)\).
Then the following proposition says that \(\Freq{\cD}(\ccn)[\alpha]\) is symmetric.
From this, the function \(\Freq{\cD}[\alpha]\) is called a \emph{symmetric component} of the game distribution of \(\cD\).

 \begin{proposition}
 \comment{Prop.}
\label{sec:orga300eb6}
\label{org3399301}
If \(\cD\) is an \(S(\cct, \cck, \ccv)\) with \(\cck = \cct + 1\), then
\[
 \Freq~{\cD}(\ccn)[\alpha] = \Freq~{\cD}(\alpha - \ccn)[\alpha]
\]
for \(\ccn \in \NN\). In particular,
\[
 \Freq{\cD}(\ccn)[\alpha] = \Freq{\cD}(2 {\ccv \choose \cck} - \alpha - \ccn)[\alpha].
\]
 
\end{proposition}

\begin{proof}
 \comment{Proof.}
\label{sec:org4046dfa}
\mylink{p14-1}
It suffices to show the assertion for \(\alpha = \cca_0(\cD) + \cca_{\cck}(\cD)\),
since \(\Freq~{\cD}(\ccn)[\alpha] = \Freq~{\cD'}(\ccn)[\alpha]\) for \(\cD' \in \Orbit{\cD}\).
Let \(\Block = \Block[\cD]\) and \(\ccP \in {[\ccv] \choose \cck} \setminus \Block\).
We first show that
\begin{equation}
\label{equ:prop:symmetric-components-sum}
 \bigsize{\OutNeighbors{\ccP} \cap \Block} + \bigsize{\InNeighbors{\ccP} \cap \Block} = \cck.
\end{equation}
Let \(\ccp \in \ccP\).
Since \(\cck = \cct + 1\), there is a unique \(\ccp' \in [\ccv]\) such that
\(\ccP^{(\ccp\ \ccp')} \in \Block\). 
Note that \(\size{\OutNeighbors{\ccP} \cap \Block}\) equals the number of \(\ccp \in \ccP\) with \(\ccp' < \ccp\).
Similarly,  \(\size{\InNeighbors{\ccP} \cap \Block}\) equals the number of \(\ccp \in \ccP\) with \(\ccp' > \ccp\).
Hence \eqref{equ:prop:symmetric-components-sum} holds.

Let \(\cpi\) be the permutation \textbf{\(\cci \mapsto \ccv - 1 - \cci\) of \([\ccv]\)}.
We next show that \(\cD^\cpi \in \Orbit{\cD}[\alpha]\).
Since \(\ccp < \ccq \iff \ccp^\cpi > \ccq^\cpi\) \textbf{for \(\ccp, \ccq \in [\ccv]\)}, it follows that  \(\bigsize{\InNeighbors{\ccP} \cap \Block} = \bigsize{\OutNeighbors{\ccP^\cpi} \cap \Block^\cpi}\).
By \eqref{equ:prop:symmetric-components-sum},
\[
 \bigsize{\OutNeighbors{\ccP} \cap \Block} = i \iff \bigsize{\OutNeighbors{\ccP^\cpi} \cap \Block^\cpi} = \cck - \cci.
\]
This implies that \(\cca_i(\cD) = \cca_{\cck - i}(\cD^\cpi)\), and hence that
\(\cca_0(\cD^\cpi) + \cca_{\cck}(\cD^\cpi) = \cca_\cck(\cD) + \cca_0(\cD) = \alpha\).
Thus \(\cD^\cpi \in \Orbit{\cD}[\alpha]\).
We also see that
\[
 \cca_0(\cD) = \ccn \iff \cca_0(\cD^\cpi) = \alpha - \ccn.
\]
Therefore \(\Freq~{\cD}(\ccn)[\alpha] = \Freq~{\cD}(\alpha - \ccn)[\alpha]\).
\end{proof}

 \begin{example}
 \comment{Exm.}
\label{sec:orge09cc14}
\label{org8b68bdd}
The symmetric components \textbf{of the game distribution} of an \(S(2, 3, 9)\) are as shown in Table \ref{tab:orgfb51808}.
\vspace{-0.5em}
{\tabcolsep = 1mm
\begin{table}[H]
\caption{\label{tab:orgfb51808}The symmetric components of \textbf{an} \(S(2, 3, 9)\).}
\centering
\begin{tabular}{c|ccccccccccccc|c}
\(\ccn\) & 68 & 69 & 70 & 71 & 72 & 73 & 74 & 75 & 76 & 77 & 78 & 79 & 80 & Total\\
\hline
\rule[0mm]{0mm}{3.6mm} \(\Freq{\cD}(\ccn)[24]\) & 1 & 1 & 1 & 1 & 1 & 1 & 1 & 1 & 1 &  &  &  &  & 9\\
\(\Freq{\cD}(\ccn)[22]\) &  & 5 & 6 & 6 & 7 & 6 & 7 & 6 & 6 & 5 &  &  &  & 54\\
\(\Freq{\cD}(\ccn)[20]\) &  &  & 9 & 14 & 24 & 23 & 22 & 23 & 24 & 14 & 9 &  &  & 162\\
\(\Freq{\cD}(\ccn)[18]\) &  &  &  & 15 & 28 & 40 & 49 & 54 & 49 & 40 & 28 & 15 &  & 318\\
\(\Freq{\cD}(\ccn)[16]\) &  &  &  &  & 17 & 24 & 37 & 45 & 51 & 45 & 37 & 24 & 17 & 297\\
\hline
\(\Freq{\cD}(\ccn)\) & 1 & 6 & 16 & 36 & 77 & 94 & 116 & 129 & 131 & 104 & 74 & 39 & 17 & 840\\
\end{tabular}
\end{table}
}\normalsize
 
\end{example}

 \begin{example}
 \comment{Exm.}
\label{sec:org5c4af94}
\label{org0e7433a}
\mylink{p14-15}
As shown in Table \ref{tab:orgb398118}, \textbf{the game distribution of} an \(S(2, 3, 7)\) has a unique symmetric component. 
This property characterizes projective Steiner triple systems as we will see in Theorem \ref{org1705e9e}.

\vspace{-0.5em}
{\tabcolsep = 1mm
\begin{table}[H]
\caption{\label{tab:orgb398118}The symmetric component of \textbf{an} \(S(2, 3, 7)\).}
\centering
\begin{tabular}{c|cccccccc|c}
\(n\) & 28 & 29 & 30 & 31 & 32 & 33 & 34 & 35 & Total\\
\hline
\rule[0mm]{0mm}{3.6mm} \(\Freq{\cD}(\ccn)[7]\) & 1 & 3 & 5 & 6 & 6 & 5 & 3 & 1 & 30\\
\hline
\(\Freq{\cD}(\ccn)\) & 1 & 3 & 5 & 6 & 6 & 5 & 3 & 1 & 30\\
\end{tabular}
\end{table}
}\normalsize
 
\end{example}

 \begin{example}
 \comment{Exm.}
\label{sec:org49b7a92}
\label{org6f257a1}
\mylink{p14-17}
Let \(\cD\) be an \(S(5, 6, 12)\). 
Then
\begin{equation}
\label{equ:exm-5-6-12-comp}
 \cca_i(\cD) = \cca_{6 - i}(\cD),
\end{equation}
so the symmetric components of \textbf{\(\Freq{\cD}\)} are as shown in Table \ref{tab:org29d0a6e}.
We can prove \eqref{equ:exm-5-6-12-comp} as follows.
If \(\ccB \in \Block[\cD]\), then \(\overline{\ccB} \in \Block[\cD]\), where \(\overline{\ccB} = [12] \setminus \ccB\).
This implies that if \(\ccP \in {[12] \choose 6} \setminus \Block[\cD]\) and \(\ccP^{(\ccp\ \ccq)} \in \Block[\cD]\),
then \(\overline{\ccP} \in {[12] \choose 6} \setminus \Block[\cD]\) and \(\overline{\ccP}^{(\ccp\ \ccq)} = \overline{\ccP^{(\ccp\ \ccq)}} \in \Block[\cD]\).
For example, if \(\cD = \shD\) and \(\ccP = \set{1, 2, 3, 4, 5, 11}\), then \(\ccP^{(5\ 0)} = \set{1, 2, 3, 4, 0, 11} \in \Block[\cD]\) and \(\overline{\ccP}^{(5\ 0)} = \set{0, 6, 7, 8, 9, 10}^{(5\ 0)} = \set{5, 6, 7, 8, 9, 10} \in \Block[\cD]\).
Note that if \(\ccP^{(\ccp\ \ccq)} \in \OutNeighbors{\ccP}\), then \(\overline{\ccP}^{(\ccp\ \ccq)} \in \InNeighbors{\overline{\ccP}}\).
It follows that if \(\size{\OutNeighbors{\ccP} \cap \Block[\cD]} = i\), then \(\size{\OutNeighbors{\overline{\ccP}} \cap \Block[\cD]} = 6 - i\).
Hence \(\cca_i(\cD) = \cca_{6 - i}(\cD)\), so its symmetric components are as shown in Table \ref{tab:org29d0a6e}.
The symmetric components of \textbf{the game distribution of an} \(S(4, 5, 11)\) are also shown in Table \ref{tab:org1ec7151}.
\vspace{-0.5em}
{\tabcolsep = 1mm
\begin{table}[H]
\caption{\label{tab:org29d0a6e}The symmetric components of \textbf{an} \(\ccS(5, 6, 12)\).}
\centering
\begin{tabular}{c|cccccccccccc|c}
\(\ccn\) & 905 & 906 & 907 & 908 & 909 & 910 & 911 & 912 & 913 & 914 & 915 & 916 & Total\\
\hline
\rule[0mm]{0mm}{3.6mm} \(\Freq{\cD}(\ccn)[38]\) & 1 &  &  &  &  &  &  &  &  &  &  &  & 1\\
\(\Freq{\cD}(\ccn)[36]\) &  & 10 &  &  &  &  &  &  &  &  &  &  & 10\\
\(\vdots\) &  &  &  &  &  & \(\dotsm\) & \(\dotsm\) &  &  &  &  &  & \(\vdots\)\\
\(\Freq{\cD}(\ccn)[18]\) &  &  &  &  &  &  &  &  &  &  & 115 &  & 115\\
\(\Freq{\cD}(\ccn)[16]\) &  &  &  &  &  &  &  &  &  &  &  & 1 & 1\\
\hline
\(\Freq{\cD}(\ccn)\) & 1 & 10 & 42 & 150 & 351 & 650 & 1012 & 1237 & 939 & 532 & 115 & 1 & 5040\\
\end{tabular}
\end{table}
}
{\tabcolsep = 1mm
\begin{table}[H]
\caption{\label{tab:org1ec7151}The symmetric components of \textbf{an} \(S(4, 5, 11)\).}
\centering
\begin{tabular}{c|cccccccccccc|c}
\(\ccn\) & 443 & 444 & 445 & 446 & 447 & 448 & 449 & 450 & 451 & 452 & 453 & 454 & Total\\
\hline
\rule[0mm]{0mm}{3.6mm} \(\Freq{\cD}(\ccn)[29]\) & 1 & 1 & 3 &  &  &  &  & 3 & 1 & 1 &  &  & 10\\
\(\Freq{\cD}(\ccn)[28]\) &  & 7 & 7 & 12 & 8 & 8 & 8 & 12 & 7 & 7 &  &  & 76\\
\(\Freq{\cD}(\ccn)[27]\) &  & 2 & 25 & 62 & 62 & 77 & 77 & 62 & 62 & 25 & 2 &  & 456\\
\(\Freq{\cD}(\ccn)[26]\) &  &  & 7 & 44 & 106 & 155 & 212 & 155 & 106 & 44 & 7 &  & 836\\
\(\Freq{\cD}(\ccn)[25]\) &  &  &  & 31 & 167 & 220 & 345 & 345 & 220 & 167 & 31 &  & 1526\\
\(\Freq{\cD}(\ccn)[24]\) &  &  &  & 1 & 8 & 140 & 259 & 402 & 259 & 140 & 8 & 1 & 1218\\
\(\Freq{\cD}(\ccn)[23]\) &  &  &  &  &  & 50 & 101 & 216 & 216 & 101 & 50 &  & 734\\
\(\Freq{\cD}(\ccn)[22]\) &  &  &  &  &  &  & 10 & 36 & 58 & 36 & 10 &  & 150\\
\(\Freq{\cD}(\ccn)[21]\) &  &  &  &  &  &  &  & 6 & 9 & 9 & 6 &  & 30\\
\(\Freq{\cD}(\ccn)[20]\) &  &  &  &  &  &  &  &  & 1 & 2 & 1 &  & 4\\
\hline
\(\Freq{\cD}(\ccn)\) & 1 & 10 & 42 & 150 & 351 & 650 & 1012 & 1237 & 939 & 532 & 115 & 1 & 5040\\
\end{tabular}
\end{table}
}
 
\end{example}

\section{A Characterization of Projective Steiner Triple Systems}
\label{sec:org261f1be}
\label{orge91698d}

 \begin{theorem}
 \comment{Thm.}
\label{sec:org1bec4ad}
\label{org1705e9e}
\mylink{p14-thm}If \(\cD\) is an \(S(2, 3, \ccv)\),
then \(\cD\) is projective if and only if
\textbf{its game distribution} has a unique symmetric component, that is, \(\size{\ccs(\cD)} = 1\).
 
\end{theorem}

\comment{comment}
\label{sec:org3de278a}
To prove Theorem \ref{org1705e9e}, we show two lemmas.

Let \(\cD\) be an \(S(2, 3, \ccv)\).
The following lemma enables us to calculate \(\cca_0(\cD) + \cca_3(\cD)\).
Recall that \(\lambda_i\), defined in \eqref{equ:lambdai}, is the number of blocks of \(\cD\) containing \textbf{\(i\) given points}.

 \begin{lemma}
 \comment{Lem.}
\label{sec:org4044dc6}
\label{org0f0d991}
Let \(\cD\) be a \(\design{\cct}{\ccv, \cck, \lambda}\) with block set \(\Block\).
\begin{enumerate}
\item If \(\cct \ge 1\) and \(\Block\) is \textbf{an independent set} in \(\Wel{\ccv}{\cck}\), then
\[
 \displaystyle \sum_{i = 0}^k i \cca_i(\cD) = \frac{\lambda_0 \cck(\ccv - \cck)}{2}.
 \]
\item If \(\cD\) is an \(S(2, 3, \ccv)\), then
\[
  \cca_0(\cD) + \cca_3(\cD) = \sum_{\set{B, C} \in {\Block \choose 2}} \ccI(\ccB, \ccC) - \frac{\ccv(\ccv - 1)(\ccv - 3)}{12},
 \]
where \(\ccI(\ccB, \ccC) = \size{\InNeighbors{\ccB} \cap \InNeighbors{\ccC}}\).
\end{enumerate}
 
\end{lemma}
\begin{proof}
 \comment{Proof.}
\label{sec:orgf0b3dc0}
\noindent
(1) Let
\[
 \IntersectPositions{\cD}{\cci} = \biggfamily{\ccP \in {[\ccv] \choose \cck} \setminus \Block : \bigsize{\OutNeighbors{\ccP} \cap \Block} = i}. \mylink{p16-6}
\]
Then \(\cca_i(\cD)= \size{\IntersectPositions{\cD}{\cci}}\).
Since \(\Block\) is \textbf{an independent set}, 
we see that
\(\bigcup_{i = 0}^{\cck} \IntersectPositions{\cD}{\cci} = {[\ccv] \choose \cck} \setminus \Block\)
and
\begin{align*}
 \sum_{i = 0}^{\cck} i \cca_i(\cD) &= \sum_{i = 0}^{\cck} \Bigsize{\Bigfamily{(\ccP, \ccB) : \ccP \in \IntersectPositions{\cD}{\cci},\ \ccB \in \OutNeighbors{\ccP} \cap \Block}} \\
 &= \biggsize{\biggfamily{(\ccP, \ccB) : \ccP \in {[\ccv] \choose \cck} \setminus \Block,\ \ccB \in \OutNeighbors{\ccP} \cap \Block} } \\
 &= \Bigsize{\Bigfamily{\rule[0mm]{0mm}{3.4mm}(\ccP, \ccB) : \ccB \in \Block,\ \ccP \in \InNeighbors{\ccB}}} \\
 &= \sum_{\ccB \in \Block} \bigsize{\InNeighbors{\ccB}}.\mylink{p16-12}
\end{align*} 
Let \(\ccB \in \Block\) and write \(\ccB = \set{\ccb_1, \dotsc, \ccb_k}\) with \(\ccb_1 < \dotsb < \ccb_k\).
Then \(\ccP \in \InNeighbors{\ccB}\) if and only if \(\ccP = \ccB^{(\ccb_i\ \ccp)}\) for some \(\ccb_i\) and \(\ccp \not \in \ccB\) with \(\ccp > \ccb_i\).
The number of such \(\ccp\) equals \(\ccv - 1 - \ccb_i - (\cck - i)\). Hence
\begin{align*}
 \size{\InNeighbors{\ccB}} &= \sum_{i = 1}^\cck \bigl(\ccv - 1 - \ccb_i - (\cck - i)\bigr)  
 = \cck(\ccv - 1) - \sum_{\ccb \in \ccB} \ccb - \frac{\cck(\cck - 1)}{2}.
\end{align*}
Since
\begin{align*}
 \sum_{\ccB \in \Block} \sum_{\ccb \in \ccB} \ccb &= (0 + 1 + \dotsb + \ccv - 1) \lambda_1 \\
 &= \frac{v(v - 1)}{2} \frac{{v - 1 \choose t - 1}}{{k - 1 \choose t - 1}} \lambda \\
 &= \frac{k(v - 1)}{2} \frac{{v \choose t}}{{k \choose t}} \lambda \\
 &= \frac{\cck(\ccv - 1) \lambda_0}{2},
\end{align*}
it follows that
\begin{align*}
 \sum_{\ccB \in \Block} \bigsize{\InNeighbors{\ccB}} &= \lambda_0 \left(\cck(\ccv - 1) - \frac{\cck(\ccv - 1)}{2} - \frac{\cck(\cck - 1)}{2} \right) 
 = \frac{\lambda_0 \cck(\ccv - \cck)}{2}.
\end{align*}

\vspace{0.5em}
\noindent
(2) Let \(\cca_i = \cca_i(\cD)\).
We first show that \mylink{p17-3}
\begin{equation}
\label{equ:lem:a0a3}
 \sum_{\cci = 0}^3 {i \choose 2} \cca_i = \cca_2 + 3 \cca_3 = \sum_{\set{\ccB, \ccC} \in {\Block \choose 2}} \ccI(\ccB, \ccC). 
\end{equation}
By the definition of \(\cca_i\),
\begin{align*}
\sum_{\cci = 0}^3 {i \choose 2} \cca_i &= 
 \sum_{\cci = 0}^3 \size{\Set{\bigl(\ccP, \set{\ccB, \ccC}\bigr) : \ccP \in \IntersectPositions{\cD}{\cci},\ \set{\ccB, \ccC} \in {\OutNeighbors{\ccP} \cap \Block \choose 2}}} \\
 &= \size{\Set{\bigl(\ccP, \set{\ccB, \ccC}\bigr) : \ccP \in {[\ccv] \choose 3} \setminus \Block,\ \set{\ccB, \ccC} \in {\OutNeighbors{\ccP} \cap \Block \choose 2}}} \\
 &= \size{\Set{\bigl(\ccP, \set{\ccB, \ccC}\bigr) : \set{\ccB, \ccC} \in {\Block \choose 2},\ \ccP \in \InNeighbors{\ccB} \cap \InNeighbors{\ccC}}} \\
 &= \sum_{\set{B, C} \in {\Block \choose 2}} \ccI(\ccB, \ccC).
\end{align*}

We now calculate \(\cca_0 + \cca_3\). 
Note that \(\lambda_0 = \ccv(\ccv - 1)/6\).
By (1),
\[
 \sum_{\cci = 0}^3 \cci \cca_i = \cca_1 + 2 \cca_2 + 3 \cca_3 = \frac{\lambda_0 3(\ccv - 3)}{2} = \frac{\ccv (\ccv - 1)(\ccv - 3)}{4}.
\]
Since
\[
 \sum_{\cci = 0}^3  \cca_i = \cca_0 + \cca_1 + \cca_2 + \cca_3 = {\ccv \choose 3} - \lambda_0 = \frac{\ccv (\ccv - 1)(\ccv - 3)}{6},
\]
it follows that
\[
 \sum_{\cci = 0}^3 \cci \cca_i - \sum_{\cci = 0}^3 \cca_i = - \cca_0 + \cca_2 + 2 \cca_3 = \frac{\ccv (\ccv - 1)(\ccv - 3)}{12}. \mylink{p17-14}
\]
Therefore
\[
 \cca_0 + \cca_3 + \frac{\ccv (\ccv - 1)(\ccv - 3)}{12} = \cca_2 + 3 \cca_3 = \sum_{\set{B, C} \in {\Block \choose 2}} \ccI(\ccB, \ccC).
\]
\end{proof}

\comment{connect}
\label{sec:org8f7d48c}
\mylink{p17-16}
From Lemma \ref{org0f0d991}, to calculate \(\cca_0 + \cca_3\), it is enough to compute \(\sum I(B, C)\).
To do this, the following easy lemma will be used.

 \begin{lemma}
 \comment{Lem.}
\label{sec:org4c5f92b}
\label{orgcbfd905}
\mylink{p26-lem}Let \(\ccB\) and \(\ccC\) be two distinct blocks of an \(S(2, 3, \ccv)\).
\begin{enumerate}
\item If \(\ccI(\ccB, \ccC) \neq 0\), then \(\size{\ccB \cap \ccC} = 1\).
\item If \(\ccB = \set{\ccx, \ccb, \ccb'}\), \(\ccC' = \set{\ccx, \ccc, \ccc'}\), \(\ccb < \ccb'\), and \(\ccb < \ccc < \ccc'\), then
\begin{equation}
\label{equ:lem:intersection}
 \ccI(\ccB, \ccC) = [\ccc < \ccb'] + [\ccc' < \ccb'],
\end{equation}
where \([\ ]\) is the Iverson bracket notation, that is, \([\ccS] \Teq 1\) if a statement \(\ccS\) holds, and \([\ccS] = 0\) otherwise.
\item If
\[
  \ccI(\ccB, \ccC) \neq \ccI(\ccB^{(i\ i + 1)}, \ccC^{(i\ i + 1)}),
 \]
then \textbf{\(\set{\ccB, \ccC} = \bigfamily{\set{\cci, \ccx, \ccy}, \set{\cci + 1, \ccx, \ccz}}\) for some \(\ccx, \ccy, \ccz \in [\ccv]\) with \(\ccy \neq \ccz\)}.
\end{enumerate}
 
\end{lemma}

\begin{proof}
 \comment{Proof.}
\label{sec:org0751f33}
\noindent
(1) If \(\ccB \cap \ccC = \emptyset\), then \(\ccI(\ccB, \ccC) = 0\).
If \(\size{\ccB \cap \ccC} \ge 2\), then \(\ccB = \ccC\). Hence \(\size{\ccB \cap \ccC} = 1\).

\noindent
(2) Since \(\ccb < \ccb'\) and \(\ccb < \ccc < \ccc'\), it follows that
\[
 \InNeighbors{\ccB} \cap \InNeighbors{\ccC} = \begin{cases}
 \emptyset & \tif \ccb' < \ccc, \\
 \bigfamily{\set{\ccx, \ccb', \ccc'} } & \tif \ccc < \ccb' < \ccc', \\
 \bigfamily{\set{\ccx, \ccb', \ccc'}, \set{\ccx, \ccb', \ccc}} & \tif \ccc' < \ccb'.
 \end{cases}
\]
Thus \eqref{equ:lem:intersection} holds.

\noindent
(3) By (1), if \(\size{\ccB \cap \ccC} \neq 1\) then \(\ccI(\ccB, \ccC) = \ccI\bigl(\ccB^{(i\ i + 1)}, \ccC^{(i\ i + 1)}\bigr)\).
Let \(\ccB = \set{\ccx, \ccb, \ccb'}\) and \(\ccC = \set{\ccx, \ccc, \ccc'}\).
We may assume that \(\ccb < \ccb'\) and \(\ccb < \ccc < \ccc'\).
If
\[
 [\ccy^{(i\ i + 1)} < \ccz^{(i\ i + 1)}] \neq [\ccy <\ccz], \mylink{p18-8}
\]
then \(\set{\ccy, \ccz} = \set{i, i + 1}\).
It follows from (2) that \(\set{i, i + 1} \in \bigfamily{\set{\ccb, \ccc}, \set{\ccb', \ccc}, \set{\ccb', \ccc'}}\).
Hence \textbf{\(\set{\ccB, \ccC} = \bigfamily{\set{\ccx, \cci, \ccy}, \set{\ccx, \cci + 1, \ccz}}\) for some \(\ccy, \ccz \in [\ccv]\) with \(\ccy \neq \ccz\)}.
\end{proof}

\begin{proof}[\yon Proof of Theorem \textrm {\ref{org1705e9e}} \yoff]
 \comment{Proof. [\yon Proof of Theorem \textrm {\ref{org1705e9e}} \yoff]}
\label{sec:org2f5cd9e}
Let \(\cD\) be an \(S(2, 3, \ccv)\) with block set \(\Block\).

We first show that if \(\cD\) is projective,
then \(\size{\ccs(\cD)} = 1\).
By Lemma \ref{orgcbfd905},
\[
 \sum_{\set{\ccB, \ccC} \in {\Block \choose 2}} \ccI(\ccB, \ccC) = \sum_{\set{\ccB, \ccC} \in {\Block \choose 2},\ \size{\ccB \cap \ccC} = 1} \hspace{-2em} \ccI(\ccB, \ccC).
\]
Let \(\set{\ccx, \ccb, \ccb'}\) and \(\set{\ccx, \ccc, \ccc'}\) be two distinct blocks of \(\cD\).
We may assume that \(\ccb < \ccb'\) and \(\ccb < \ccc < \ccc'\). \mylink{p18-16}
Since \(\cD\) is projective, there are \(\ccy, \ccz \in [\ccv]\) such that \(\set{\ccy, \ccb, \ccc}\), \(\set{\ccy, \ccb', \ccc'}\), \(\set{\ccz, \ccb, \ccc'}\), and \(\set{\ccz, \ccb', \ccc}\) are blocks.
Thus we obtain the following triple of pairs of blocks:
\[
 \Bigfamily{\, \bigfamily{\set{\ccx, \ccb, \ccb'}, \set{\ccx, \ccc, \ccc'}}, \quad \bigfamily{\set{\ccy, \ccb, \ccc}, \set{\ccy, \ccb', \ccc'}}, \quad \bigfamily{\set{\ccz, \ccb, \ccc'}, \set{\ccz, \ccb', \ccc}}\, }.
\]
By Lemma \ref{orgcbfd905},
\begin{equation}
\label{equ:six-blocks-sum}
\begin{split}
 \ccI\bigl(\set{\ccx, \ccb, \ccb'}&,  \set{\ccx, \ccc, \ccc'}\bigr) + \ccI\bigl(\set{\ccy, \ccb, \ccc},  \set{\ccy, \ccb', \ccc'}\bigr) + \ccI\bigl(\set{\ccz, \ccb, \ccc'}, \set{\ccz, \ccb', \ccc}\bigr) \\
 &= [\ccc < \ccb'] + [\ccc' < \ccb'] + [\ccb' < \ccc] + [\ccc' < \ccc] + [\ccb' < \ccc'] + [\ccc < \ccc'] = 3.
\end{split}
\end{equation}
Note that every pair \(\set{\ccB, \ccC}\) of blocks with \(\size{\ccB \cap \ccC} = 1\) belongs to exactly one such triple and there are \(\ccv {\lambda_1 \choose 2}/3\) triples.
Hence
\[
 \sum_{\set{\ccB, \ccC} \in {\ccB \choose 2},\ \size{\ccB \cap \ccC} = 1} \hspace{-2em} \ccI(\ccB, \ccC) = \frac{\ccv {\lambda_1 \choose 2}}{3} {3}  = \frac{\ccv (\ccv - 1)(\ccv - 3)}{8}.
\]
By Lemma \ref{org0f0d991},
\[
 \cca_0(\cD) + \cca_3(\cD) = \sum_{\set{B, C} \in {\Block \choose 2}} \ccI(\ccB, \ccC) - \frac{\ccv(\ccv - 1)(\ccv - 3)}{12} = \frac{\ccv (\ccv - 1)(\ccv - 3)}{24}.
\]
Therefore \(\ccs(\cD) = \set{\ccv(\ccv - 1)(\ccv - 3)/24}\).

\mylink{p17-11}We next show that if \(\cD\) is not projective,
then \(\size{\ccs(\cD)} > 1\).
Note that \(\ccv \ge 9\).
Since \(\cD\) is not projective, 
it follows from Remark \ref{orgde64e0f} that we may assume that
\[
 \set{0, \ccv - \ymod{2}, 2}, \set{1, \ccv - \ymod{2}, 3}, \set{0, \ccv - \ymod{3}, 3}, \set{1, \ccv - \ymod{3}, 4} \in \Block
\]
by replacing \(\cD\) with some \(\cD' \in \Orbit{\cD}\) if necessary (see Figure \ref{fig:org63c30f4}).
\begin{figure}[htbp]
\centering
\includegraphics[scale=1.0]{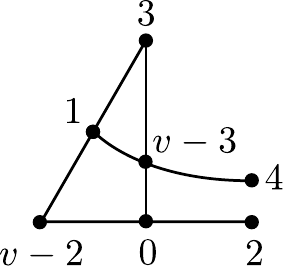}
\caption{\label{fig:org63c30f4}The four blocks of \(\cD\).}
\end{figure}
\textbf{We may also assume that \(\set{0, 1, \ccv - 1} \in \Block\)}.
\yon
We show that
\begin{equation}
\label{equ:thm-projective-claim}
 \sum_{\set{\ccB, \ccC} \in {\Block \choose 2}} \ccI(\ccB, \ccC) - \ccI(\ccB, \ccC)^{(0\ 1)} \neq 0 \quad \tor \sum_{\set{\ccB, \ccC} \in {\Block \choose 2}} \ccI(\ccB, \ccC)^{(2\ 3)} - \ccI(\ccB, \ccC)^{(2\ 3)(0\ 1)} \neq 0,
\end{equation}
where \(\ccI(\ccB, \ccC)^{(0\ 1)} = \ccI(\ccB^{(0\ 1)}, \ccC^{(0\ 1)})\).
If \(\set{\ccB, \ccC} \in {\Block \choose 2}\) and \(\ccI(\ccB, \ccC) \neq \ccI(\ccB, \ccC)^{(0\ 1)}\),
then, by Lemma \ref{orgcbfd905}, 
\[
 \set{\ccB, \ccC} = \bigfamily{\set{0, \ccx, \ccy}, \set{1, \ccx, \ccz}}
\]
for some \(\ccx, \ccy, \ccz \in [\ccv]\) with \(\ccy \neq \ccz\);
note that \(\ccx \neq \ccv - 1\) since \(\set{0, 1, \ccv - 1} \in \Block\).
For \(\ccx \in \set{2, 3, \dotsc, \ccv - \ymod{2}}\), let \(\ccB_{\ccx} = \set{0, \ccx, \ccy_{\ccx}} \in \Block\) and \(\ccC_x = \set{1, \ccx, \ccz_{\ccx}} \in \Block\). 
By assumption,
\yoff
\begin{equation}
\label{equ:theorem-assumption}
\begin{split}
 &\ccB_{\ccv - \ymod{2}} = \set{0, \ccv - \ymod{2}, 2},\ \ccC_{\ccv - \ymod{2}} = \set{1, \ccv - \ymod{2}, 3},\\
 &\ccB_{\ccv - \ymod{3}} = \set{0, \ccv - \ymod{3}, 3},\ \ccC_{\ccv - \ymod{3}} = \set{1, \ccv - \ymod{3}, 4}.
\end{split}
\end{equation}
\yon
We see that
\begin{equation}
\label{equ:sumBC}
 \sum_{\set{\ccB, \ccC} \in {\Block \choose 2}} \ccI(\ccB, \ccC) - \ccI(\ccB, \ccC)^{(0\ 1)} = \sum_{\ccx = 2}^{\ccv - 2} \ccI(\ccB_\ccx, \ccC_\ccx) - \ccI(\ccB_\ccx, \ccC_\ccx)^{(0\ 1)}.
\end{equation}
Similarly, by Lemma \ref{orgcbfd905}, if \(\set{\ccB, \ccC} \in {\Block \choose 2}\) and \(\ccI(\ccB, \ccC)^{(2\ 3)} \neq \ccI(\ccB, \ccC)^{(2\ 3)(0\ 1)}\), then 
\[
 \bigset{\ccB^{(2\ 3)}, \ccC^{(2\ 3)}} = \bigfamily{\set{0, \ccx, \ccy}^{(2\ 3)}, \set{1, \ccx, \ccz}^{(2\ 3)}}
\]
for some \(\ccx, \ccy, \ccz \in [\ccv]\) with \(\ccy \neq \ccz\),
and hence 
\yoff
\(\set{\ccB, \ccC} = \set{\ccB_{\ymod{\ccx}}, \ccC_{\ymod{\ccx}}}\).
\yon
Therefore
\begin{equation}
\label{equ:sumBC23}
 \sum_{\set{\ccB, \ccC} \in {\Block \choose 2}} \ccI(\ccB, \ccC)^{(2\ 3)} - \ccI(\ccB, \ccC)^{(2\ 3)(0\ 1)} = \sum_{\ccx = 2}^{\ccv - 2} \ccI(\ccB_\ccx, \ccC_\ccx)^{(2\ 3)} - \ccI(\ccB_\ccx, \ccC_\ccx)^{(2\ 3)(0\ 1)}.
\end{equation}
\yoff
If \(\ccI(\ccB_{\ymod{\ccx}}, \ccC_{\ymod{\ccx}}) \neq \ccI(\ccB_{\ymod{\ccx}}, \ccC_{\ymod{\ccx}})^{\ymod{^{(2\ 3)}}}\),
then \textbf{\(\set{\ccB_{\ccx}, \ccC_{\ccx}} = \bigfamily{\set{2, \ccx, \ccu}, \set{3, \ccx, \ccw}}\) for some \(\ccu, \ccw \in [\ccv]\)},
so \(\ymod{\ccx} = \ccv - 2\) by \eqref{equ:theorem-assumption}. 
Similarly, if 
\(\ccI(\ccB_{\ymod{\ccx}}, \ccC_{\ymod{\ccx}})^{\ymod{(0\ 1)}} \neq \ccI(\ccB_{\ymod{\ccx}}, \ccC_{\ymod{\ccx}})^{\ymod{(2\ 3)(0\ 1)}}\), 
then \(\ymod{\ccx} = \ccv - \ymod{2}\).
Let \(\ccB_{\ymod{o}} = \ccB_{\ccv - \ymod{2}}\) and \(\ccC_{\ymod{o}} = \ccC_{\ccv - \ymod{2}}\).
By \eqref{equ:sumBC} and \eqref{equ:sumBC23},
\yon
\begin{align*}
 &\sum_{\set{\ccB, \ccC} \in {\Block \choose 2}} \ccI(\ccB, \ccC) - \ccI(\ccB, \ccC)^{(0\ 1)} - \left(\ccI(\ccB, \ccC)^{(2\ 3)} - \ccI(\ccB, \ccC)^{(2\ 3)(0\ 1)}\right)\\
 &=  \sum_{\ccx = 2}^{\ccv - 2} \ccI(\ccB_\ccx, \ccC_\ccx) - \ccI(\ccB_\ccx, \ccC_\ccx)^{(0\ 1)} - \left(\ccI(\ccB_\ccx, \ccC_\ccx)^{(2\ 3)} - \ccI(\ccB_\ccx, \ccC_\ccx)^{(2\ 3)(0\ 1)}\right)\\
 &=  \sum_{\ccx = 2}^{\ccv - 2} \ccI(\ccB_\ccx, \ccC_\ccx) - \ccI(\ccB_\ccx, \ccC_\ccx)^{(2\ 3)} - \left(\ccI(\ccB_\ccx, \ccC_\ccx)^{(0\ 1)} - \ccI(\ccB_\ccx, \ccC_\ccx)^{(2\ 3)(0\ 1)}\right)\\
 &= \ccI(\ccB_{\ymod{o}}, \ccC_{\ymod{o}}) - \ccI(\ccB_{\ymod{o}}, \ccC_{\ymod{o}})^{(2\ 3)} - \left(\ccI(\ccB_{\ymod{o}}, \ccC_{\ymod{o}})^{(0\ 1)} - \ccI(\ccB_{\ymod{o}}, \ccC_{\ymod{o}})^{(2\ 3)(0\ 1)}\right).
\end{align*}
\yoff
Since
\begin{align*}
 \ccI(\ccB_{\ymod{o}}, \ccC_{\ymod{o}}) & =  \ccI(\set{0, \ccv - \ymod{2}, 2}, \set{1, \ccv - \ymod{2}, 3}) = 1, \\
\ccI(\ccB_{\ymod{o}}, \ccC_{\ymod{o}})^{\ymod{(2\ 3)}}  &= \ccI(\set{0, \ccv - \ymod{2}, 3}, \set{1, \ccv - \ymod{2}, 2}) = 2, \\
\ccI(\ccB_{\ymod{o}}, \ccC_{\ymod{o}})^{\ymod{(0\ 1)}}  &= \ccI(\set{1, \ccv - \ymod{2}, 2}, \set{0, \ccv - \ymod{2}, 3}) = 2, \\
\ccI(\ccB_{\ymod{o}}, \ccC_{\ymod{o}})^{\ymod{(2\ 3)(0\ 1)}} &= \ccI(\set{1, \ccv - \ymod{2}, 3}, \set{0, \ccv - \ymod{2}, 2}) = 1,
\end{align*}
\yon
we see that
\[
 \sum_{\set{\ccB, \ccC} \in {\Block \choose 2}} \ccI(\ccB, \ccC) - \ccI(\ccB, \ccC)^{(0\ 1)} - \left(\ccI(\ccB, \ccC)^{(2\ 3)} - \ccI(\ccB, \ccC)^{(2\ 3)(0\ 1)}\right) = 1 - 2 - (2 - 1) = - 2.
\]
Therefore \eqref{equ:thm-projective-claim} holds. It follows from Lemma \ref{org0f0d991} that \(\size{\ccs(\cD)} > 1\).
\yoff
\end{proof}

\comment{bibliography}
\label{sec:orgeffe3fe}


\begin{thebibliography}{10}
\bibitem{baartmans-binary-1996}
A.~Baartmans, I.~Landjev, and V.~D. Tonchev.
\newblock On the binary codes of {{Steiner}} triple systems.
\newblock {\em Des. Codes Cryptogr.}, 8(1):29--43, 1996.

\bibitem{berlekamp-Winning-2001}
E.~R. Berlekamp, J.~H. Conway, and R.~K. Guy.
\newblock {\em Winning ways for your mathematical plays}.
\newblock {A.K. Peters}, Natick, \ymod{MA}, 2nd edition, 2001.

\bibitem{colbourn-triple-1999}
C.~J. Colbourn and A.~Rosa.
\newblock {\em Triple systems}.
\newblock {Clarendon Press}, \ymod{Oxford}, 1999.

\bibitem{cole-complete-1917}
F.~N. Cole, L.~D. Cummings, and H.~S. White.
\newblock The complete enumeration of triad systems in 15 elements.
\newblock {\em Proc. Natl. Acad. Sci. USA}, 3:197--199, 1917.

\bibitem{conway-numbers-2001}
J.~H. Conway.
\newblock {\em On numbers and games}.
\newblock {A.K. Peters}, Natick, \ymod{MA}, 2nd edition, 2001.

\bibitem{conway-lexicographic-1986}
J.~H. Conway and N.~J.~A. Sloane.
\newblock Lexicographic codes: error-correcting codes from game
  theory.
\newblock {\em IEEE Trans. Inform. Theory}, 32(3):337--48, 1986.

\bibitem{conway-sphere-2013}
J.~H. Conway and N.~J.~A. Sloane.
\newblock {\em Sphere packings, lattices and groups}, volume 290.
\newblock {Springer}, \ymod{New York, NY, 3rd edition, 1999}.

\bibitem{depasquale-sui-1899}
V.~{de Pasquale}.
\newblock {Sui sistemi ternari di 13 elementi}.
\newblock {\em Rendiconti Reale Istituto Lombardo di Scienze e Lettere},
  2(32):213--221, 1899.

\bibitem{dehon-designs-1977}
M.~Dehon.
\newblock {Designs et hyperplans}.
\newblock {\em J. Combin. Theory Ser. A}, 23(3):264--274,
  1977.

\bibitem{diaconis-mathematics-1983}
P.~Diaconis, R.~L. Graham, and W.~M. Kantor.
\newblock The mathematics of perfect shuffles.
\newblock {\em Adv. Appl. Math.}, 4(2):175--196, 1983.

\bibitem{doyen-ranks-1978}
J.~Doyen, X.~Hubaut, and M.~Vandensavel.
\newblock Ranks of incidence matrices of {{Steiner}} triple systems.
\newblock {\em Math. Z.}, 163(3):251--259, 1978.

\yon
\bibitem{fraenkel-games-2019}
A.~S. Fraenkel, U.~Larsson.
\newblock {{Games on arbitrarily large rats and playability}}.
\newblock {\em Integers}, 19:1--29(\#G04), 2019.
\yoff

\bibitem{kahane-hexad-2001}
J.~Kahane and A.~J. Ryba.
\newblock The {{Hexad game}}.
\newblock {\em Electron. J. Combin.}, 8(2):1--9(\#R11),
  2001.

\bibitem{kirkman-problem-1847}
T.~P. Kirkman.
\newblock On a problem in combinations.
\newblock {\em Cambridge and Dublin Mathematical Journal}, 2:191--204, 1847.

\bibitem{moore-concerning-1893}
E.~H. Moore.
\newblock Concerning triple systems.
\newblock {\em Math. Ann.}, 43(2):271--285, 1893.

\bibitem{veblen-projective-1916}
O.~Veblen and J.~W. Young.
\newblock {\em Projective geometry}.
\newblock {Ginn \ymod{and Company}}, {Boston}, 1916.

\bibitem{witt-5fach-1937}
E.~Witt.
\newblock {Die 5-fach transitiven gruppen von Mathieu}.
\newblock {\em Abh. Math. Semin. Univ. Hambg.}, 12(1):256--264, 1937.

\bibitem{witt-ueber-1937}
E.~Witt.
\newblock {{\"U}ber Steinersche systeme}.
\newblock {\em Abh. Math. Semin. Univ. Hambg.}, 12:265--275, 1937.

\end{thebibliography}
\end{document}